\documentclass{article}
\usepackage[utf8]{inputenc}
\usepackage[left=3cm, right=3cm, top=3cm, bottom=3cm]{geometry}
\usepackage[colorlinks=true, allcolors=blue]{hyperref}

\usepackage{bm,  mathrsfs, amsmath, amssymb, amsthm, amscd, amsfonts,
graphicx,epsfig,latexsym,layout,fancyhdr, colonequals, enumerate,
hyperref, fontenc, graphics, verbatim, psfrag, lipsum, comment, color}
  \usepackage{comment}
  \usepackage{stmaryrd}
\usepackage[OT1]{fontenc}
\usepackage{subcaption}
\usepackage{doi}
\usepackage{mathtools}  
\usepackage{authblk} 
\usepackage{hyperref} 
\usepackage[numbers]{natbib}

\usepackage{tikz}

\usepackage[ruled,vlined]{algorithm2e}
\SetKwFor{ForAll}{for all}{do}{end}

\newcommand{\C}{\mathbb{C}}
\newcommand{\R}{\mathbb{R}}
\renewcommand{\H}{\mathbb{H}}
\newcommand{\Z}{\mathbb{Z}}
\newcommand{\K}{K}
\newcommand{\A}{\K}
\newcommand{\M}{\mathrm{M}}

\newcommand{\calA}{\mathcal{A}}
\newcommand{\calB}{\mathcal{B}}
\newcommand{\calC}{\mathcal{C}}

\newcommand{\calT}{\mathcal{T}}

\newcommand{\calX}{\mathcal{X}}
\newcommand{\calF}{\mathcal{F}}

\theoremstyle{definition}
\newtheorem{definition}{Definition}[section]
\newtheorem{example}[definition]{Example}
\newtheorem{remark}[definition]{Remark}

\theoremstyle{plain}
\newtheorem{theorem}[definition]{Theorem}
\newtheorem{corollary}[definition]{Corollary}
\newtheorem{lemma}[definition]{Lemma}
\newtheorem{proposition}[definition]{Proposition}

\title{\textbf{Singular value decompositions over $\ast$-algebras with applications to tensor decompositions}}

\author[1]{Paulina Hoyos}
\author[2]{Joe Kileel}

\affil[1]{Department of Mathematics, University of Texas at Austin, Austin, TX, USA\protect\\
\href{mailto:paulinah@utexas.edu}{\texttt{paulinah@utexas.edu}}}
\affil[2]{Department of Mathematics and Oden Institute for Computational Engineering and Sciences,\protect\\
University of Texas at Austin, Austin, TX, USA\protect\\
\href{mailto:jkileel@math.utexas.edu}{\texttt{jkileel@math.utexas.edu}}}
\date{\vspace{-3mm}October 29, 2026}

\begin{document}

\maketitle
\vspace{-5mm}
\begin{abstract}
\noindent The singular value decomposition (SVD) provides optimal low-rank approximations of real and complex matrices by the Eckart--Young theorem, and SVD-based tensor decomposition algorithms such as the higher-order singular value decomposition (HOSVD) and the tensor train singular value decomposition (TT-SVD) inherit quasi-optimality guarantees from it. We extend the SVD to matrices over a finite-dimensional $\ast$-algebra $K$.
Our main contribution is a characterization of the algebras $K$ for which every matrix admits an SVD, and of the inner products on $K$ for which the Eckart–Young theorem holds.
 
We introduce a notion of ordered SVD, show that ordered SVDs reveal the rank of a matrix over $K$ and yield optimal truncations, and give an explicit algorithm that computes an  ordered SVD from the structure of the algebra. We then extend the Tucker, tensor train, and tree tensor network (TTN) decompositions and their SVD-based decomposition algorithms to tensors over $K$, and show that TT-SVD retains a quasi-optimality guarantee over every such $K$, whereas HOSVD, ST-HOSVD, and TTN-SVD do so when $K$ is commutative. 
Viewed in real or complex coordinates, these constructions give novel decomposition formats of ordinary tensors.

We illustrate our new methods by compressing data from a variety of scientific applications, and
show that they can achieve lower approximation error than their classical counterparts at equal storage. \\


\end{abstract}

\section{Introduction}
The singular value decomposition (SVD) is arguably the most important matrix factorization in numerical linear algebra and data analysis. Every matrix $A \in \C^{m \times n}$ (or $\R^{m \times n}$) admits an SVD $A = U \Sigma V^H$, where $U \in \C^{m \times m}$ and $V \in \C^{n \times n}$ are unitary (or orthogonal in the real case) and $\Sigma \in \R^{m \times n}$ is diagonal with real entries $\sigma_1 \geq \sigma_2 \geq \cdots \geq \sigma_{\min(m,n)} \geq 0$ called the singular values of $A$. 
The SVD can be computed stably and efficiently, even for quite large matrices~\cite{golub2013}. 
It exposes key features of $A$.  For example, the rank of $A$ is the number of nonzero singular values, and the columns of $U$ and $V$ associated with nonzero singular values give orthonormal bases for the column and row spaces of $A$, respectively. 

The SVD also yields optimal low-rank approximations. 
By the Eckart--Young theorem~\cite{eckart1936,mirsky1960}, the truncation $A_k = \sum_{i=1}^{k} \sigma_i u_i v_i^H$, where $u_i$ and $v_i$ are the columns of $U$ and $V$, is a best rank-$k$ approximation to $A$ with respect to any unitarily invariant norm.

This result underlies many methods for compression, denoising and dimensionality reduction, including principal component analysis (PCA) 
and the compression of neural networks~\cite{xue2013, denton2014}.

Tensors, or multiway arrays, are widely used to represent data indexed by several modes~\cite{BallardKolda2025}. 
Since an order-$d$ tensor with $n$ entries per mode has $n^d$ entries, compression is essential, and   low-rank tensor decompositions such as the canonical polyadic  (CP)~\cite{hitchcock1927, carroll1970, harshman1970}, Tucker~\cite{tucker1966}, tensor train (TT)~\cite{Oseledets2011}, and more generally tree tensor networks (TTN), which include the hierarchical Tucker format
\cite{hackbusch2009new,grasedyck2010hierarchical}, represent a tensor with far fewer parameters.
Standard algorithms for Tucker, TT and TTN (HOSVD~\cite{DeLathauwer2000}, ST-HOSVD~\cite{Vannieuwenhoven2012}, TT-SVD~\cite{Oseledets2011}, and the
hierarchical SVD \cite{grasedyck2010hierarchical})
compute truncated SVDs of unfoldings of the tensor. 
An important fact is that SVD-based tensor decomposition algorithms are generally {quasi-optimal}: the error of the computed approximation is within a factor depending only on $d$ of the best possible error.  Quasi-optimality follows from the Eckart--Young theorem for matrices.

In many tensor datasets, there exist modes that do not index separate measurements but rather describe the internal structure of an object.  For example, consider the color channels of a pixel, the time samples of a signal, or the values of a function over a group. 
We argue in this paper that such data are often better viewed as a lower-order tensor, whose entries belong to an algebra $K$ rather than the real or complex numbers, where the multiplication of the algebra helps to encode the structure of the relevant object.
As one prior example in the literature, in the t-product framework~\cite{kilmer2011}, a collection of grayscale images is viewed as a matrix whose entries are tubes, multiplied via circular convolution.
As another prior example, some authors have represented 
a color image as a matrix of quaternions~\cite{sangwine1996}. 
Related frameworks, especially building upon the t-product, have also been developed~\cite{kernfeld2015, avron2025, hoyos2026}; see Section~\ref{sec:related-work}.
Each of these frameworks develops its own SVD and, in some cases, under additional conditions on the algebra, an optimality result for truncated SVDs~\cite{zhang1997, jia2019, kilmer2011, kilmer2021, hoyos2026}. 
In this work we characterize the finite-dimensional algebras $K$ for which every matrix admits an SVD and the inner products on $K$ for which the Eckart–Young theorem holds,
 recovering the SVDs above as special cases and extending partial results on the existence of SVDs~\cite{ginzberg2015} and on Eckart--Young-type theorems for tubal algebras~\cite{mor2025}.

Likewise, tensor decompositions beyond the real and complex numbers have been developed one algebra at a time: Tucker and TT decompositions have been extended to tubal algebras~\cite{wang2022hot, ahmadiasl2026} and to the quaternions~\cite{miaokou2021, YaYang2023, miao2022}. Exploiting the unifying perspective of $*$-algebras, we extend the Tucker, TT and TTN  decompositions, together with their SVD-based algorithms, to tensors over any algebra for which the Eckart--Young theorem holds, recovering these constructions as special cases. 
We show that TT-SVD retains its classical quasi-optimality guarantee for every such $K$, whereas the guarantees for HOSVD, ST-HOSVD and TTN-SVD hold when $K$ is commutative, recovering the bounds of \cite{wang2022hot}.  

\subsection{Our Contributions}
We generalize the notion of SVD to matrices over a unital $\ast$-algebra $K$, requiring the singular values to be nonnegative  elements of $K$.

In Theorem~\ref{thm: existence of SVD},
we show that a finite-dimensional unital $\ast$-algebra $K$ is very proper ($\sum_i a_i^\ast a_i =0$ implies $a_i =0$ for all $i$) if and only if every matrix over $K$ admits an SVD.  Algorithm~\ref{alg:SVD} computes an SVD of a matrix $A \in K^{m \times n}$ through the Wedderburn--Artin decomposition of $K$, at the cost of transforming the entries of $A$ plus one classical SVD in each Wedderburn--Artin component. 

We also introduce \emph{ordered} SVDs, defined through a strengthened version of the Loewner order on $K$.  In Proposition~\ref{prop: characterization of ordered SVDs} we prove that, up to conjugating each singular value by a unitary element of $K$, ordered SVDs are exactly the SVDs produced by Algorithm~\ref{alg:SVD} and have uniquely determined singular values. We define the rank of a matrix over $K$ as the least number of outer products whose sum is the matrix, and prove in Proposition~\ref{prop:rank} that rank equals the number of nonzero singular values in any ordered SVD.

Optimal truncations require compatibility  between the involution and inner product on $K$: $\langle ab, c\rangle = \langle b, a^\ast c\rangle = \langle a, c b^\ast\rangle$
for all $a, b, c \in K$. In Theorem~\ref{thm: optimality of rank-k SVD approximation},  we show that, for a finite-dimensional unital $\ast$-algebra $K$ with an inner product, the inner product satisfies these identities if and only if every matrix over $K$ has an ordered SVD such that, for every $k$, its rank-$k$ truncations are optimal in the induced Frobenius norm.

We use the SVD over $K$ to extend the Tucker, TT and TTN decompositions to tensors over $K$. 
Algorithms~\ref{alg:HOSVD} and~\ref{alg:STHOSVD} compute a Tucker decomposition, Algorithm~\ref{alg:TT-SVD} a TT decomposition, and Algorithm~\ref{alg:ttn-svd} a TTN decomposition on any rooted tree. 
When the Eckart--Young theorem holds over $K$, Theorem~\ref{cor: optimality for TT-SVD} shows that TT-SVD is quasi-optimal with factor $\sqrt{d-1}$, as in the classical case, whether or not $K$ is commutative. For HOSVD and ST-HOSVD, Theorem~\ref{cor: optimality for Tucker decompositions} gives the classical factor $\sqrt{d}$ when $K$ is commutative. Example~\ref{ex: HOSVD fails} shows failure of the HOSVD guarantee for a noncommutative algebra.

In real or complex coordinates, our constructions over $K$ 
give novel decomposition formats for ordinary tensors.  
We illustrate efficacy of the new methods by compressing
climate data, as well as data from quantum circuits and Vlasov–Poisson plasma simulations.  In these numerical tests, 
decompositions over $\star$-algebras achieve lower approximation error than their classical counterparts at equal storage.

Our framework therefore unifies disparate matrix and tensor decompositions, while yielding many new decomposition formats.  We also identify the algebraic conditions that imply approximation guarantees.

\subsection{Related Work}\label{sec:related-work}

\textbf{Tensor-tensor products.}
Kilmer and Martin~\cite{kilmer2011} introduced the t-product, which views a third-order tensor as a matrix of tubes multiplied by circular convolution. 
They also introduced the t-SVD, whose truncations are optimal in the Frobenius norm. 
The  algebra of tubes was further developed in~\cite{kilmer2013}, and the t-SVD was extended to tensors of higher order in~\cite{martin2013}.
Kernfeld et al.~\cite{kernfeld2015} replaced the discrete Fourier transform with an arbitrary invertible transform $M$. Truncations of the resulting $\star_M$-SVD are optimal when $M$ is a nonzero multiple of a unitary matrix~\cite{kilmer2021}, and Mor~\cite{mor2025} characterized the
transforms for which this holds. 
Avron and Mor~\cite{avron2025} treated these tubal algebras in a unified way. Navasca et al.~\cite{navasca2010} replaced circular convolution with convolution over an abelian group, and Hoyos et al.~\cite{hoyos2026} extended this to any finite group. 
All of these are finite-dimensional $\ast$-algebras.  So our results apply to them and recover their SVDs and optimality results (Corollaries~\ref{cor:existence}, \ref{cor:mor} and~\ref{cor:starG}), while applying to other examples as well.

Liao et al.~\cite{liao2022} defined an HOSVD over tubal algebras without error bounds, Wang and Yang~\cite{wang2022hot} extended HOSVD and ST-HOSVD to tubal tensors with $\sqrt{d}$ error bounds, and Ahmadi-Asl et al.~\cite{ahmadiasl2026} introduced a tubal tensor train. Our framework
recovers these decompositions and their error bounds (Corollaries~\ref{cor:hot-svd} and~\ref{cor:ttt}), and extends the decompositions to cover 
noncommutative algebras. 
In a different direction, Mei et al.~\cite{mei2026} use TT and hierarchical Tucker formats to accelerate
transform-based products.\\[-1.5mm]

\noindent\textbf{Quaternions.}
Zhang proved the existence of the quaternion SVD in ~\cite{zhang1997}. Algorithms for computing the quaternion SVD~\cite{sangwine2006,jia2019,liu2021} have made it a
tool in color image processing~\cite{sangwine1996}. 
Tensor decompositions over the quaternions include versions of the t-SVD~\cite{miaokou2021tsvd,qin2025}, the HOSVD~\cite{miaokou2021,YaYang2023}, Tucker and CP decompositions~\cite{flamant2024}, and the tensor train decomposition~\cite{miao2022}. 
For HOSVD, the order of the mode products matters~\cite{miaokou2021}, and its direct analogue can be far from quasi-optimal~\cite{YaYang2023}. We show that TT-SVD is quasi-optimal over $\H$ (Corollary~\ref{cor:qtt}), while Example~\ref{ex: HOSVD fails} shows that the failure of HOSVD can occur for noncommutative algebras beyond the quaternions.\\[-1.5mm]

\noindent\textbf{SVDs over general algebras.}
The closest work to ours is perhaps that of Ginzberg and Mavroyiakoumou~\cite{ginzberg2015}.  They compute SVDs over a class of
real $\ast$-algebras by a Jacobi-type algorithm, and by a reduction similar to our Algorithm~\ref{alg:SVD}. In their definition, however, the diagonal
factor $\Sigma$ may have arbitrary entries, and they do not address necessary
conditions for existence, rank, or low-rank approximation. 
The SVDs of dual complex and dual quaternion matrices~\cite{qi2022,qiluo2021} have singular values that are nonnegative only in a lexicographic order. These algebras are not very proper, so they do not admit SVDs in our sense. Yu et
al.~\cite{yu2024} introduced an SVD for tensors of reduced biquaternions, and Shirokov~\cite{shirokov2024} studies the SVD of a single element of a Clifford algebra.

\section{Matrices and tensors over a $\ast$-algebra $K$}
In this section, we review the basic definitions of algebras and $\ast$-algebras. We then introduce matrices and tensors with entries in a $\ast$-algebra, which are the main objects of this paper.

\subsection{Algebras and $\ast$-algebras}
Throughout this article, we let $F\in\{\R,\C\}$. An \textbf{algebra} over $F$ is a vector space $K$ with scalars in $F$
equipped with a multiplication operation that is
associative and bilinear: 
\begin{equation*}
    (ab)c=a(bc), \quad (\alpha a+\beta b)c=\alpha\,ac+\beta\,bc, \quad c(\alpha a+\beta b)=\alpha\,ca+\beta\,cb
\end{equation*}
for all $a,b,c\in K$ and $\alpha,\beta\in F$. A \textbf{unit} (or multiplicative identity) of $K$ is an element $1\in K$ with $1a=a1=a$ for all
$a\in K$; it is unique when it exists, and in that case we say that $K$ is \textbf{unital}. The \textbf{center} of $K$ is the collection of all elements that commute with every element of $K$; such elements are called \textbf{central}.
As a vector space over $F$, we consider the 
\textbf{dimension} of $K$, denoted $\mathrm{dim}_F(K)$, and 
an \textbf{inner product}  $\langle\cdot,\cdot\rangle\colon K\times K\to F$ on $K$ with associated norm $\|a\|=\langle a,a\rangle^{1/2}$.

For example, the complex numbers $\C$ form a unital algebra over $\C$ of dimension $1$, with
inner product $\langle z,w\rangle=z\overline{w}$ for all $z,w\in\C$. 
The \textbf{quaternions} are the $4$-dimensional real algebra
\[
  \H=\{\,q_0+q_1\,i+q_2\,j+q_3\,k \;:\; q_0,q_1,q_2,q_3\in\R\,\}
\]
with multiplication determined by $i^2=j^2=k^2=ijk=-1$. 
The conjugate of $q=q_0+q_1\,i+q_2\,j+q_3\,k$ is $\bar q=q_0-q_1\,i-q_2\,j-q_3\,k$; 
it satisfies $\overline{pq}=\overline{q}\,\overline{p}$ and induces an inner product on $\H$ given by 
\[\langle p, q \rangle = \operatorname{Re}(p \overline{q})=p_0q_0+p_1q_1+p_2q_2+p_3q_3,\]
where $\operatorname{Re}(q)=q_0$ denotes the real part of $q$. 
The corresponding norm is $\|q\|^2 = q \overline{q} =q_0^2+q_1^2+q_2^2+q_3^2$, and every nonzero $q$ is invertible with $q^{-1}=\bar q/(q\overline{q})$.  The quaternions are not commutative, for instance $ij=k=-ji$; their center is given by the real numbers.

The following definition abstracts the properties of complex and quaternionic conjugation to a general algebra $K$.
\begin{definition}
A \textbf{$\ast$-algebra} over $F$ is an algebra $K$ over $F$ together with a map
$^\ast \colon K\to K$, $a\mapsto a^*$, called an \textbf{involution}, such that for all $a,b\in K$ and $\alpha, \beta\in F$,
\[
  (\alpha a+ \beta b)^*=\overline{\alpha }a^*+ \overline{\beta} b^*,\qquad (\alpha a)^*=\bar\alpha\,a^*,\qquad
  (ab)^*=b^*a^*,\qquad (a^*)^*=a,
\]
where $\overline{\alpha}$ denotes complex conjugation (trivial if $F = \R$).
An element $a\in K$ is \textbf{self-adjoint} if $a^*=a$. If $K$ is unital, an element $u\in K$ is \textbf{unitary} if $u^*u=uu^*=1$. If $L$ is another $\ast$-algebra over $F$, a \textbf{$\ast$-isomorphism}
$\varphi\colon K\to L$ is a linear bijection that preserves products and the involution, i.e., $\varphi(ab)=\varphi(a)\varphi(b)$ and $\varphi(a^*)=\varphi(a)^*$ for all
$a,b\in K$. 
\end{definition}

Standard examples of $\ast$-algebras are given by real or complex matrices. 
\begin{example}[Matrix algebras]
The space $\mathrm{M}_n(\R)$ of real $n\times n$ matrices with the transpose $A\mapsto A^T$ as
involution is a unital $\ast$-algebra over $\R$ of dimension $n^2$. Its self-adjoint
elements are the symmetric matrices and its unitary elements are the orthogonal matrices. Its center  is the collection of scalar matrices $\alpha I_n$ with $\alpha \in \R$. 

Similarly, the space $\mathrm{M}_n(\C)$ of complex $n\times n$ matrices with 
the conjugate transpose $A\mapsto A^H$ as involution, is a unital $\ast$-algebra over $\C$
of dimension $n^2$. Its self-adjoint elements are the Hermitian matrices and its unitary
elements are the unitary matrices. Its center  is the collection of scalar matrices $\alpha I_n$ with  $\alpha \in \C$. 
\end{example}

Here are some further examples of $\ast$-algebras.

\begin{example}[Tubal algebras]\label{ex: tubal algebras}
Let $M\in\C^{p\times p}$ be invertible. The $\star_M$-product~\cite{kernfeld2015} of tubes
$a,b\in\C^p$ is
\[
  a\star_M b = M^{-1}\bigl((Ma)\odot(Mb)\bigr),
\]
where $\odot$ is the entrywise product. With the involution $a^*=M^{-1}\,\overline{Ma}$ and
the Euclidean inner product, $\C^p$ is a commutative unital $\ast$-algebra over $\C$ of
dimension $p$, with unit $M^{-1}\mathbf{1}$, where $\mathbf{1}=(1,\dots,1)$. If every row of
$M$ is real or the conjugate of exactly one other row, then the real tubes
$\mathbb{R}^p$ are closed under $\star_M$ and the
involution~\cite[Lemma~2]{avron2025}, and form a commutative unital
$*$-algebra over $\mathbb{R}$. When $M$ is the discrete Fourier transform,
this gives the t-product~\cite{kilmer2011}.
\end{example}

\begin{example}[Group algebras] \label{ex: group algebras}
Let $G$ be a finite group. The group algebra $F[G]$ consists of functions $a\colon G\to F$,
with the convolution product, involution and inner product given by
\[
  (ab)(g)=\sum_{h\in G}a(h)\,b(h^{-1}g),\qquad a^*(g)=\overline{a(g^{-1})},\qquad
  \langle a,b\rangle=\sum_{g\in G}a(g)\,\overline{b(g)}.
\]
It is a unital $\ast$-algebra of dimension $|G|$, with unit the indicator function of the
identity, and it is commutative exactly when $G$ is abelian.  In the $\star_G$-product~\cite{hoyos2026}, $G=G_1\times\cdots\times G_q$ is a product of finite groups, and a tensor 
of order $q+2$ is viewed as a matrix over $F[G_1 \times\cdots\times G_q]$.
\end{example}

\begin{example}[Bandlimited functions on a compact group]\label{ex:bandlimited}
Let $G$ be a compact topological group with Haar measure normalized so that $\int_G \mathrm{d}g = 1$. For an irreducible unitary representation $\rho$ of $G$ of dimension $d_\rho$, the Fourier coefficient of $a \in L^2(G)$ is $\hat a(\rho) = \int_G a(g)\, \rho(g)\, \mathrm{d}g$. Fix a finite set $\Lambda$ of (pairwise inequivalent) such representations, and let $K_\Lambda$ be the space of functions $a \in L^2(G)$ with $\hat a(\rho) = 0$ for $\rho \notin \Lambda$, with
\[
(ab)(g) = \int_G a(h)\, b(h^{-1}g)\, \mathrm{d}h, \qquad
a^*(g) = \overline{a(g^{-1})}, \qquad
\langle a, b \rangle = \int_G {a(g)}\, \overline{b(g)}\, \mathrm{d}g .
\]
Since $\widehat{ab}(\rho) = \hat a(\rho)\,\hat b(\rho)$ and $\widehat{a^*}(\rho) = \hat a(\rho)^{\mathrm{H}}$, $K_\Lambda$ is a unital $\ast$-algebra over $\mathbb{C}$ of dimension $\sum_{\rho \in \Lambda} d_\rho^2$, with unit $e_\Lambda(g) = \sum_{\rho \in \Lambda} d_\rho \overline{\operatorname{tr} \rho(g)}$. It is commutative exactly when $d_\rho = 1$ for all $\rho \in \Lambda$. For $G = \mathrm{SO}(2)$ and $\Lambda = \{e^{ik\theta} : |k| \le L\}$, $K_\Lambda$ consists of the trigonometric polynomials of degree at most $L$.
\end{example}

\begin{example}[Clifford algebras]\label{ex:clifford}
Let $n = p + q$. The Clifford algebra $\mathrm{Cl}_{p,q}$ is the real algebra generated by $e_1, \dots, e_n$ with $e_i^2 = 1$ for $i \le p$, $e_i^2 = -1$ for $i > p$, and $e_i e_j = -e_j e_i$ for $i \ne j$. The products $e_S = e_{s_1} \cdots e_{s_k}$ for $S = \{s_1 < \cdots < s_k\} \subseteq [n]$, with $e_\emptyset = 1$, form a basis, so $\dim_{\mathbb{R}} \mathrm{Cl}_{p,q} = 2^n$. The involution determined by $e_i^* = e_i^{-1}$ satisfies $e_S^* = e_S^{-1}$, and with the inner product $\langle a, b \rangle = \sum_S a_S b_S$ for $a = \sum_S a_S e_S$ and $b = \sum_S b_S e_S$, $\mathrm{Cl}_{p,q}$ is a unital $\ast$-algebra over $\mathbb{R}$. It is noncommutative for $n \ge 2$. For instance, $\mathrm{Cl}_{0,1} \cong \mathbb{C}$ via $e_1 \mapsto i$, and $\mathrm{Cl}_{0,2} \cong \mathbb{H}$ via $e_1 \mapsto i$, $e_2 \mapsto j$, $e_1e_2 \mapsto k$; in both cases $\ast$ is the usual conjugation.
\end{example}

\subsection{Matrices over $K$}
Given a unital $\ast$-algebra $K$ over $F$, we write $K^{m\times n}$ for the set of $m\times n$ matrices with entries in $K$. As for real and complex matrices, we write $A_{ij}\in K$ for the entry of $A \in K^{m\times n}$ in row $1 \leq i \leq m $ and column $1\leq j \leq n$. With entrywise 
addition and scalar multiplication, $K^{m\times n}$ is a vector space over $F$ of dimension $\dim_F(K^{m\times n}) = mn\dim_F(K)$.
The \textbf{product of matrices} $A\in K^{m\times r}$ and
$B\in K^{r\times n}$ is the matrix $AB\in K^{m\times n}$ given by
\[
  (AB)_{ij}=\sum_{k=1}^{r}A_{ik}B_{kj},
\]
where $A_{ik}B_{kj}$ denotes multiplication in the algebra $K$, and the order of the factors matters because $K$ need not be commutative. 
The \textbf{conjugate transpose} of $A\in K^{m\times n}$ is the matrix $A^H\in K^{n\times m}$ defined as
\[
  (A^H)_{ij}\colonequals(A_{ji})^*,
\]
where $(A_{ji})^*$ denotes the image of $A_{ji} \in K$ under the involution for $K$.
The \textbf{identity matrix} $I_n\in K^{n\times n}$ has $(I_n)_{ij}=1$ if $i=j$ and $0$
otherwise. A matrix $U\in K^{n\times n}$ is \textbf{unitary} if
\[
  U^HU=UU^H=I_n ,
\]
and we write $\mathrm{U}_n(K)$ for the set of unitary matrices in $K^{n\times n}$.
With these operations, $K^{n\times n}$ is again a unital $\ast$-algebra over $F$, with
unit $I_n$ and involution $A\mapsto A^H$; in particular $(AB)^H=B^HA^H$ whenever the
product is defined.
Given an inner product $\langle \cdot, \cdot \rangle$ for  $K$,  an inner product of matrices $A, B \in \K^{m\times n}$ is defined as
\begin{equation}\label{eq: inner product of matrix}
    \langle A,B \rangle = \sum_{i =1}^{m}\sum_{j=1}^{n} \langle A_{ij}, B_{ij} \rangle.
\end{equation}
The associated norm is $ \|A\| = \sqrt{\langle A, A \rangle},$ which we call the \textbf{Frobenius norm} of $A$.

\begin{example}[Quaternion matrices]
For $K=\H$, the algebra $\H^{n\times n}$ of quaternionic $n\times n$ matrices, with the conjugate transpose as involution, is a unital $\ast$-algebra over $\R$ of dimension $4n^2$. Its center consists of the scalar matrices $\alpha I_n$ with $\alpha\in\R$.
\end{example}

\begin{example}[Matrices of tubes]
Let $K=\C^p$ be the tubal algebra with the $\star_M$-product. A matrix $A\in K^{m\times n}$ is a tensor of size $m\times n\times p$, whose entries $A_{ij}\in\C^p$ are its tubes, and the product of matrices over $K$ is the $\star_M$-product of third-order
tensors~\cite{kernfeld2015}. For $M$ given by the discrete Fourier transform, this is the t-product~\cite{kilmer2011}.
\end{example}

\subsection{Tensors over $K$}
Given a $\ast$-algebra $K$ over $F$, we write $K^{n_1 \times \cdots \times n_d}$ for the space of $d$-way arrays of size $n_1 \times n_2 \times \cdots \times n_d$ with entries in $K$.  An element $\calA \in K^{n_1 \times \cdots \times n_d}$ is called an order-$d$ tensor, and we denote its entries as $\calA(i_1,\dots,i_d)\in K$. Addition and scalar multiplication are defined entrywise:
\begin{equation*}
    (\calA + \calB)(i_1, \dots, i_d) = \calA(i_1, \dots, i_d) + \calB(i_1, \dots, i_d), \qquad  (\alpha \calA)(i_1, \dots, i_d) = \alpha \calA(i_1, \dots, i_d)
\end{equation*}
for $\calA, \calB \in  K^{n_1 \times \cdots \times n_d}$ and $\alpha \in F$. These operations make $K^{n_1 \times \cdots \times n_d}$ a vector space over $F$ of dimension $\dim_F(K^{n_1\times\cdots\times n_d}) = n_1\cdots n_d\dim_F(K)$.
Given an inner product $\langle \cdot, \cdot \rangle$ for $K$,  the inner product of tensors $\calA, \calB \in \K^{n_1 \times \cdots \times n_d}$ is defined as
\begin{equation}\label{eq: inner product of tensor}
    \langle \calA, \calB \rangle = \sum_{i_1 =1}^{n_1}\cdots \sum_{i_d =1}^{n_d} \langle \calA(i_1, \dots, i_d), \calB(i_1, \dots, i_d) \rangle.
\end{equation}
The  associated norm is $\|\calA\| = \sqrt{\langle \calA, \calA \rangle}$, which we call the \textbf{Frobenius norm} of $\calA$.

\begin{example}[Tensors over tubal and group algebras]
If $K=\C^p$ is a tubal algebra, a tensor $\calA\in K^{n_1\times\cdots\times n_d}$ is a complex tensor of size $n_1\times\cdots\times n_d\times p$, whose last mode holds the coordinates of the entries in $K$. 
Similarly, for $K=F[G_1\times\cdots\times G_q]$, an order-$d$ tensor over $K$ is an order-$(d+q)$ tensor over $F$. 
Conversely, any tensor over $F$ becomes a tensor over such a $K$ by choosing which of its modes to treat as the
coordinates of $K$; this choice is a modeling decision, and in Section~\ref{sec:experiments} we make it according to the symmetries of the data.
\end{example}

\section{Singular value decompositions}

In this section, we extend the SVD to matrices over a finite-dimensional unital $\ast$-algebra $K$. We discuss both existence and uniqueness; in particular, we give an explicit algorithm for computing the SVD. We also show how the rank of a matrix over $K$ is related to the singular values of its SVDs.

\subsection{Existence of SVDs}
In the classical SVD, the singular values are nonnegative real numbers; over a $\ast$-algebra this role is played by the following elements.
\begin{definition}
    The \textbf{nonnegative cone} of $K$ is $K^+ \colonequals \{\, a^*a \;:\; a \in K \,\}.$
    Its elements are called \textbf{nonnegative}.
\end{definition}
In particular, every nonnegative element of $K$ is self-adjoint. For $K=\R$, $\C$ or $\H$, the nonnegative elements are the nonnegative real numbers, since $a^*a=|a|^2$, and for $K=\M_n(D)$ with $D\in\{\R,\C,\H\}$ they are the positive semidefinite matrices.

\begin{definition}
    A matrix $A \in K^{m \times n}$ admits a \textbf{singular value decomposition (SVD)} if there exist unitary matrices $U \in K^{m \times m}$ and $V \in K^{n \times n}$ and a diagonal matrix $\Sigma \in K^{m \times n}$ with diagonal entries $\Sigma_{ii} \in K^+$ such that
    \begin{equation}
        A = U \Sigma V^H.
    \end{equation}
    The $\Sigma_{ii}$ are the \textbf{singular values} of this SVD.
\end{definition}
For $K=\R$ or $\C$ this is the classical SVD (without ordering of the singular values), and for $K=\H$ it is the quaternion SVD~\cite{zhang1997}. Not every unital $\ast$-algebra admits SVDs, as the following example shows.
\begin{example}[A matrix without an SVD]
\label{ex: no SVD}
Let $K=\C$, viewed as a real algebra with the identity as a non-standard involution. The column
\[
  A=\begin{pmatrix}1\\ i\end{pmatrix}\in K^{2\times1}
\]
satisfies $A^HA=1^\ast1+i^\ast i= 1\cdot1 + i \cdot i =0$, where $A^H$ denotes the conjugate transpose with respect to the non-standard involution.  If $A=U\Sigma V^H$ were an SVD with $U \in \mathrm{U}_2(K)$, $V = (v) \in U_1(K)$, and $\Sigma=(\sigma,0)^T$, then $0=A^HA=v\,\sigma^2v$, so $\sigma^2=0$ and hence $\sigma=0$, which forces $A=0$. Thus $A$ admits no SVD.
\end{example}

The problem in Example~\ref{ex: no SVD} is that $1^*1+i^*i=0$, although $1$ and $i$ are nonzero. Algebras in which this cannot happen are called very proper, or formally real~\cite{CHMN2015}.
\begin{definition}\label{ax: condition for existence of SVD}
    A $\ast$-algebra $K$ is \textbf{very proper} if $a_1, \dots, a_N \in K$ and $\sum_{i=1}^N a_i^\ast a_i = 0 $ implies 
    \begin{equation}
        a_1 = \cdots = a_N = 0.
    \end{equation}
\end{definition}

The algebras $\R$, $\C$ and $\H$ are very proper, since $\sum_i a_i^*a_i = \sum_i |a_i|^2$ vanishes only if all $a_i = 0$. Similarly, $\M_n(D)$ is very proper for $D \in \{\R, \C, \H\}$, since $\operatorname{tr}\bigl(\sum_i A_i^H A_i\bigr) = \sum_i \|A_i\|_F^2$ vanishes only if all $A_i = 0$. The algebra of Example~\ref{ex: no SVD} is proper ($a^*a=0$ implies $a=0$) but it is not very proper; so the condition cannot be restricted to a single term.

\begin{lemma}\label{lem: structure thm for very proper algebras}
Let $K$ be a finite-dimensional very proper $\ast$-algebra. Then there exists a $\ast$-isomorphism
\begin{equation}\label{eq: isomorphism of K and matrix algebras}
    \pi \colon K \to \bigoplus_{\rho=1}^\ell \mathrm{M}_{d_\rho}(D_\rho),
\end{equation}
where $D_\rho\in\{\R,\C,\H\}$ if $F=\R$ and $D_\rho=\C$ if $F=\C$, and the involution in each summand $\mathrm{M}_{d_\rho}(D_\rho)$ is given by the conjugate transpose. 
\end{lemma}
See \cite[Prop. 2.7]{CHMN2015} for a proof. We call $\pi$ a \textbf{structure isomorphism} of $K$. It is a Wedderburn--Artin isomorphism of $K$ with the additional condition that $\pi$ respects the involution on $K$, i.e., $\pi(a^\ast) = \bigoplus_\rho \pi_\rho(a)^H$, where $\pi_\rho$ denotes the $\rho$-th component of $\pi$.

The following theorem shows that being very proper characterizes the algebras over which every matrix admits an SVD.
\begin{theorem}[Necessary and sufficient condition for existence of the SVD]\label{thm: existence of SVD}
    A finite-dimensional unital $\ast$-algebra $K$ is very proper if and only if every matrix with entries in $K$ admits a singular value decomposition.
\end{theorem}
\begin{proof}
Assume that $K$ is very proper, and let $\pi$ be a structure isomorphism of $K$ as given  by Lemma~\ref{lem: structure thm for very proper algebras}. Using $\pi$, Algorithm~\ref{alg:SVD} in Section~\ref{sec: algorithm for svd} constructs an SVD of any matrix $A \in K^{m \times n}$.

Conversely, assume that every matrix with entries in $K$ admits a singular value decomposition. Fix an integer $N \geq 1$ and elements $a_1, \dots, a_N \in K$ such that $\sum_{i=1}^N a_i^\ast a_i = 0$. Form the column vector
 \[ A = \begin{pmatrix}
     a_1 \\
     \vdots \\
     a_N
 \end{pmatrix} \in K^{N \times 1}.  \]
By hypothesis $A = U \Sigma V^H$ with $U \in \mathrm{U}_N(K)$, $V = (v) \in \mathrm{U}_1(K)$ and $\Sigma = (\sigma, 0, \dots, 0)^T$ with $\sigma \in K^+$. Since $U^H U = I_N$ and $\sigma \in K^+$ is self-adjoint, we have that
 \[ 0 = \sum_{i=1}^N a_i^\ast a_i = A^HA =  V\Sigma^H U^H U \Sigma V^H =  v \sigma^2 v^\ast. \]
Multiplying on the left by $v^\ast$ and on the right by $v$ gives $\sigma^2 =0$. As we prove below, 
$K^+$ contains no nonzero nilpotent element, so in particular $\sigma^2 =0$ implies $\sigma =0$.
It follows that $\Sigma \in K^{N \times 1}$ is the zero matrix and thus $A = U \Sigma V^H = 0$, that is, $a_1 = \cdots = a_N =0$ and $K$ is a very proper algebra.

Suppose now that $K^{+}$ contained a nonzero nilpotent, and choose $b \in K^{+}$ nonzero and nilpotent with $\dim_F bK$ maximal; this is
possible since these dimensions are bounded by $\dim_F K$. 
Write $b = a^{*}a$ with $a \in K$ and apply the existence hypothesis to $(a) \in K^{1 \times 1}$: there are unitary $u, v \in K$  and $c \in K^{+}$ such that $a = u\,c\,v^{*}$, so 
\begin{equation}\label{eq: b conjugate c}
     b \;=\; a^{*}a \;=\; v\,c^{*}u^{*}u\,c\,v^{*} \;=\; v\,c^{2}\,v^{*} .
\end{equation}
Then $c \neq 0$, since otherwise $b = 0$; and $c$ is nilpotent, since $c^{2n} = v^\ast\,b^{n}\,v$ for every $n \geq 1$ by Eq.~\eqref{eq: b conjugate c}. It follows that  $c^{2}K \subsetneq cK$, and hence $\dim_F c^{2}K < \dim_F cK$.
Indeed $c^2K \subseteq cK$ since $cK \subseteq K$, but if $c^{2}K = cK$, multiplying repeatedly on the left by $c$ gives  $c^{n}K = cK$ for every $n \geq 1$, so $c^{n} = 0$ for some $n$ forces $cK = 0$ and hence $c = c \cdot 1 = 0,$ contradicting $c\neq 0$.
Finally, since $v$ is unitary, $v^*K=K$, so $bK=v\,c^2v^*K=v\,(c^2K)$, and left multiplication by $v$ is a bijection of $K$; hence
\[
  \dim_F bK \;=\; \dim_F c^{2}K \;<\; \dim_F cK ,
\]
contradicting maximality. We conclude that $K^{+}$ contains no nonzero nilpotent element.
\end{proof}

Several frameworks construct an SVD for matrices over a particular algebra,
each with its own argument. Theorem~\ref{thm: existence of SVD} recovers these
existence results at once.

\begin{corollary}\label{cor:existence}
Every matrix admits an SVD over each of the following algebras: tubes with
the $\star_{\mathbf{M}}$-product~\cite{kernfeld2015,kilmer2021}, the group
algebra $F[G]$ of a finite group~\cite{hoyos2026}, quaternion tubes with
the $\star_{QT}$-product~\cite[Thm.~1]{miaokou2021tsvd}, and tubes of reduced
biquaternions with the $\ast_{\mathrm{Ht}}$-product~\cite[Thm.~3.1]{yu2024}.
\end{corollary}

\begin{proof}
By Theorem~\ref{thm: existence of SVD}, it suffices to show that each algebra is
very proper. For tubes with the $\star_{\mathbf{M}}$-, $\star_{QT}$- and
$\ast_{\mathrm{Ht}}$-products, the transform gives a $*$-isomorphism onto a
direct sum of copies of $\mathbb{R}$, $\mathbb{C}$ or $\mathbb{H}$, which is very proper. For the $\ast_{\mathrm{Ht}}$-product, the transform is
preceded by the splitting $q\mapsto(q_{(c),1},q_{(c),2})$ of each reduced
biquaternion into two complex numbers~\cite{yu2024}. For $F[G]$, the value of
$\sum_ia_i^*a_i$ at the identity is $\sum_i\|a_i\|^2$.
\end{proof}

\subsection{An algorithm for computing the SVD}\label{sec: algorithm for svd}
We now give an explicit algorithm for computing the SVD of matrices over $K$, which completes the proof of Theorem~\ref{thm: existence of SVD}. 

In what follows, the unital $\ast$-algebra $K$ is also finite-dimensional and very proper. Let $\pi$ be a structure isomorphism of $K$ as in Eq.~\eqref{eq: isomorphism of K and matrix algebras}, and write $\pi_\rho \colon K \to \mathrm{M}_{d_\rho}(D_\rho)$ for its $\rho$-th component.
Applying $\pi$
entrywise to matrices over $K$ gives an $F$-linear isomorphism
\begin{align}\label{eq: isomorphism for matrices over K}
\begin{split}
    \Pi \colon K^{m \times n}
      &\;\xrightarrow{\ \sim\ }\;
      \bigoplus_{\rho=1}^{\ell}\mathrm{M}_{m d_\rho \times n d_\rho}(D_\rho) \\
    A
      &\;\xmapsto{\hphantom{\ \sim\ }}\;
      \widehat{A} \colonequals \bigoplus_{\rho=1}^{\ell}\widehat{A}_\rho ,
\end{split}
\end{align}
where  $\widehat{A}_\rho$ is the $m d_\rho \times n d_\rho$ block matrix over $D_\rho$ whose $(i,j)$ block is 
\begin{equation}
  \big(\widehat{A}_\rho\big)_{ij} = \pi_\rho(A_{ij}) \in \mathrm{M}_{d_\rho}(D_\rho),
  \qquad 1 \le i \le m, \ 1 \le j \le n .
\end{equation}
We write $\Pi_\rho$ for the $\rho$-th component of $\Pi$ and call $\Pi_\rho(A)=\widehat{A}_\rho$ the $\rho$-th component of $A$. Since the construction does not depend on the matrix sizes, we write $\Pi$ for every $m,n\ge 1$.
Because $\pi$ is a $\ast$-isomorphism with $\pi(1) = 1$, $\Pi$ respects products, conjugate transposes, and identity matrices: $\Pi(AB) = \Pi(A)\,\Pi(B)$ whenever the matrix product $AB$ is defined, 
$\Pi(A^{H}) = \Pi(A)^{H}$, and $\Pi(I) = I$.

The isomorphism $\Pi$ reduces the SVD of $A$ to $\ell$ independent classical SVDs. 
Each component $\widehat{A}_\rho$ is an ordinary matrix over $D_\rho \in \{\R, \C, \mathbb{H}\}$, so it has an SVD: 
there exist unitary matrices
$\widehat{U}_\rho \in \mathrm{M}_{m d_\rho}(D_\rho)$ and
$\widehat{V}_\rho \in \mathrm{M}_{n d_\rho}(D_\rho)$, and a diagonal matrix
$\widehat{\Sigma}_\rho \in \mathrm{M}_{m d_\rho \times n d_\rho}(D_\rho)$ with
real nonnegative entries arranged in nonincreasing order along the diagonal, such that
\begin{equation*}
    \widehat{A}_\rho = \widehat{U}_\rho \widehat{\Sigma}_\rho \widehat{V}_\rho^{H}.
\end{equation*}
The matrices $U \colonequals \Pi^{-1}\big(\bigoplus_\rho \widehat{U}_\rho\big)$ and
$V \colonequals \Pi^{-1}\big(\bigoplus_\rho \widehat{V}_\rho\big)$ are unitary over $K$, since
$\Pi$ respects products and conjugate transposes.
Moreover, $\Sigma \colonequals \Pi^{-1}\big(\bigoplus_\rho \widehat{\Sigma}_\rho\big)$
is diagonal with entries in $K^+$. Indeed, since each $\widehat{\Sigma}_\rho$ is
diagonal, $\pi_\rho(\Sigma_{ij}) = (\widehat{\Sigma}_\rho)_{ij} = 0$ for $i \neq j$,
so $\Sigma_{ij} = 0$ by injectivity of $\pi$. Each block $\pi_\rho(\Sigma_{ii})$ is
diagonal with nonnegative entries, so it has a diagonal square root $Y_\rho$ with
nonnegative entries. Setting $t \colonequals \pi^{-1}\big(\bigoplus_\rho Y_\rho\big)$
gives $\pi(t^*t) = \bigoplus_\rho Y_\rho^2 = \pi(\Sigma_{ii})$, hence
$\Sigma_{ii} = t^*t \in K^+$.
The resulting triple $(U, \Sigma, V)$ is an SVD of $A$, since 
\[
  \Pi_\rho\big(U \Sigma V^{H}\big)
  = \widehat{U}_\rho \widehat{\Sigma}_\rho \widehat{V}_\rho^{H}
  = \widehat{A}_\rho  = \Pi_\rho(A) 
\]
for every $\rho$, so $U\Sigma V^{H} = A$ by injectivity of $\Pi$. This construction is summarized in
Algorithm~\ref{alg:SVD}. 

\begin{algorithm}[h!]
\caption{$\mathrm{SVD}_K$}\label{alg:SVD}
\KwIn{$A \in K^{m \times n}$}
\KwOut{$U \in K^{m \times m}$ and $V \in K^{n \times n}$ unitary, and $\Sigma \in K^{m\times n}$ diagonal with $\Sigma_{ii} \in K^{+}$, such that $A = U\Sigma V^H$.}
\ForAll{$\rho = 1, \dots, \ell$  \textnormal{\textbf{ in parallel} }}{
  $\widehat{A}_\rho \gets \Pi_\rho(A)$\;
  $[\widehat{U}_\rho,\, \widehat{\Sigma}_\rho,\, \widehat{V}_\rho]
     \gets \mathrm{svd}\big(\widehat{A}_\rho\big)$\;
}
$U \gets \Pi^{-1}\!\Big(\bigoplus_{\rho=1}^{\ell}\widehat{U}_\rho\Big)$\;
$\Sigma \gets \Pi^{-1}\!\Big(\bigoplus_{\rho=1}^{\ell}\widehat{\Sigma}_\rho\Big)$\;
$V \gets \Pi^{-1}\!\Big(\bigoplus_{\rho=1}^{\ell}\widehat{V}_\rho\Big)$\;
\Return{$U, \Sigma, V$}
\end{algorithm}

We call an SVD computed by Algorithm~\ref{alg:SVD} an \textbf{SVD induced by $\pi$}. For $K=\R$, $\C$ or $\H$, $\pi$ is the identity and Algorithm~\ref{alg:SVD} gives the classical or quaternion SVD. For tubes with the $\star_M$-product, every $d_\rho=1$, $\Pi$ applies $M$ to the tubes, and Algorithm~\ref{alg:SVD} is the $\star_M$-SVD of~\cite{kilmer2021}. For tensors under the $\star_G$-product,  $\pi$ is the Fourier transform on $G$: its components are indexed by the irreducible unitary representations $\rho$ of $G$, and $d_\rho$ is the dimension of $\rho$. 

\paragraph{Computational complexity of $\mathrm{SVD}_K$.}
Assume $m \ge n$, and write $n_K \colonequals \dim_F K$ and $d_{\max} \colonequals \max_\rho d_\rho$. We count the thin SVD, with $U \in K^{m\times n}$ and $\Sigma \in K^{n\times n}$.
Algorithm~\ref{alg:SVD} applies $\pi$ to the $mn$ entries of $A$ and $\pi^{-1}$ to the $O(mn)$ entries of $U$, $\Sigma$ and $V$. Each application is an $F$-linear map on a space of dimension $n_K$, so it costs $O(n_K^2)$. The algorithm also computes $\ell$ independent classical SVDs, one for each component $\widehat A_\rho$. Since $\widehat{A}_\rho$ is an $md_\rho \times nd_\rho$ matrix over $D_\rho$, computing its SVD costs $O(mn^2 d_\rho^3)$. The total cost is
\[
  O\Big(mn\,n_K^2 \;+\; mn^{2}\sum_{\rho=1}^{\ell} d_\rho^{3}\Big)
\]
operations in $F$. Since $\sum_\rho d_\rho^2 \le n_K$, the second term is at most $O(mn^2\,d_{\max}\,n_K)$, and it dominates when $n$ is large compared with $n_K$. When $\pi$ is a fast transform, such as the FFT for the t-product, the first term drops to $O(mn\,n_K \log n_K)$. 

\subsection{Uniqueness of ordered SVDs}
We compare consecutive singular values through central elements of $K$, using the partial order on $K$ given by $a\succeq b$ if $a-b\in K^+$. This is known as the Loewner order on $K$. 
\begin{definition}
    Given $A \in K^{m \times n}$, let $s = \min(m, n)$. 
    We say that an SVD $A = U \Sigma V^H$ is \textbf{ordered} if, for all $1 \leq i < s$, there exist self-adjoint central elements $c_i \in K$ such that
    \begin{equation}
         \Sigma_{ii} 
\succeq c_i \succeq \Sigma_{i+1, i+1}. 
    \end{equation}
\end{definition}
For $K=\R$, $\C$ or $\H$, the self-adjoint central elements are the real numbers, and the condition reduces to the classical ordering $\Sigma_{11}\ge\Sigma_{22}\ge\cdots\geq 0$.

The next result characterizes ordered SVDs in terms of a structure isomorphism of $K$.

\begin{proposition}\label{prop: characterization of ordered SVDs}
Let $\pi$ be a structure isomorphism of $K$, and write $s = \min(m,n)$ for a given matrix $A \in K^{m\times n}$. The following statements are equivalent.
    \begin{enumerate}
        \item $A = U \Sigma V^H$ is an ordered SVD.
        \item The SVD $A = U \Sigma V^H$  satisfies
        \begin{equation}\label{eq: ordering of singular values under pi}
            \sigma_{\min} (\pi_\rho (\Sigma_{ii})) \geq \sigma_{\max} (\pi_\rho(\Sigma_{i+1, i+1}))
        \end{equation}
        for all $1 \leq i < s$ and all components $\rho$, where $\sigma_{\min}$ and $\sigma_{\max}$ denote the smallest and largest singular values of a matrix over $D_\rho \in \{\R, \C, \H\}$.
        \item There exist unitary elements $w_1, \dots, w_s \in K$ such that $A = (UW_m)(W_m^H \Sigma W_n)(VW_n)^H$ is an SVD induced by $\pi$, where $W_k = \mathrm{diag}(w_1, \dots, w_s, 1, \dots, 1) \in K^{k \times k}$ for $k = m,n$.
    \end{enumerate}
\end{proposition}

In words, for any structure isomorphism $\pi$, ordered SVDs are exactly the SVDs induced by $\pi$, up to conjugating each singular value by a unitary element of $K$. In particular, the singular values of the ordered SVDs of a matrix are unique up to unitary conjugation. We prove Proposition~\ref{prop: characterization of ordered SVDs} in Appendix~\ref{app: proof for ordered SVDs}.

\begin{example}
    Take $K = \M_2(\R)$ and let $\pi \colon K \to \M_2(\R)$ be the identity. Consider $A \in K^{2\times 2}$ given by
    \begin{equation*}
    A = \begin{pmatrix}
         \begin{pmatrix} 3 & 0 \\ 0 & 2 \end{pmatrix} & \begin{pmatrix} 0&0\\0&0\end{pmatrix}\\ \vspace{-2mm}\\
          \begin{pmatrix} 0&0\\0&0\end{pmatrix} & \begin{pmatrix} 1/2&1/2\\1/2&1/2
          \end{pmatrix}
    \end{pmatrix}.
    \end{equation*}
The diagonal entries of $A$ are nonnegative, since $A_{11} = a_1^T a_1$ and $A_{22} = a_2^T a_2$ with
\begin{equation*}
   a_1 = \begin{pmatrix} \sqrt{3} & 0 \\ 0 & \sqrt{2} \end{pmatrix}, \quad
   a_2 = \frac{1}{\sqrt{2}}\begin{pmatrix} 1&1\\0&0 \end{pmatrix}.
\end{equation*}
Hence $U = V = I_2$ and $\Sigma = A$ give an SVD of $A$. This is an ordered SVD, since  $\sigma_{\min}\pi(\Sigma_{11}) = 2 > 1 = \sigma_{\max}(\pi(\Sigma_{22}))$. However, because $\pi(\Sigma_{22})$ is not diagonal, this is not an SVD induced by $\pi$. 

Consider the unitary matrix $W = \mathrm{diag}(w_1, w_2)$ with
\begin{equation*}
   w_1 = \begin{pmatrix} 1 & 0\\ 0&1 \end{pmatrix}, \quad
   w_2 = \frac{1}{\sqrt{2}} \begin{pmatrix} 1&1\\ 1&-1 \end{pmatrix}.
\end{equation*}
Then $W^H\Sigma W$ is diagonal with
\begin{equation*}
   w_1^T \Sigma_{11} w_1 = \begin{pmatrix} 3& 0\\ 0 & 2 \end{pmatrix}, \quad
   w_2^T \Sigma_{22} w_2 = \begin{pmatrix} 1& 0\\ 0 & 0 \end{pmatrix},
\end{equation*}
so $\Pi(W^H\Sigma W) = \mathrm{diag}(3,2,1,0) \in \R^{4\times 4}$, and $A = W\,(W^H\Sigma W)\,W^H$ is an SVD induced by $\pi$.
\end{example}

Not every SVD is ordered, as the next example shows.
\begin{example}
    Take $K = \M_2(\R)$ and let $\pi \colon K \to \M_2(\R)$ be the identity. Let $A \in K^{2\times 2}$ be the diagonal matrix with
    \begin{equation*}
A = \begin{pmatrix}
         \begin{pmatrix}
                4 & 0 \\
                0 & 2
            \end{pmatrix} & \begin{pmatrix} 0&0\\0&0\end{pmatrix}\\ \vspace{-2mm}\\
          \begin{pmatrix} 0&0\\0&0\end{pmatrix} & \begin{pmatrix}
                3 & 0 \\
                0 & 1
            \end{pmatrix}
    \end{pmatrix}.
    \end{equation*}
Taking $U = V = I_2$ and $\Sigma = A$ gives an SVD of $A$ that is not ordered.  Indeed, $\pi(\Sigma_{11}) = \mathrm{diag}(4,2)$ and $\pi(\Sigma_{22}) = \mathrm{diag}(3,1)$, so $\sigma_{\min}(\pi(\Sigma_{11})) < \sigma_{\max}(\pi(\Sigma_{22}))$. Note that $\Sigma_{11}\succeq\Sigma_{22}$, so comparing singular values directly is not enough to determine if an SVD is ordered or not.

An ordered SVD of $A$ is $A = U'\Sigma'U'^H$ with
\[
U' = \begin{pmatrix}
    \begin{pmatrix} 1 & 0 \\ 0 & 0 \end{pmatrix} & \begin{pmatrix} 0 & 0 \\ 1 & 0 \end{pmatrix} \\ \vspace{-2mm} \\
    \begin{pmatrix} 0 & 1 \\ 0 & 0 \end{pmatrix} & \begin{pmatrix} 0 & 0 \\ 0 & 1 \end{pmatrix}
\end{pmatrix},
\qquad
\Sigma' = \begin{pmatrix}
    \begin{pmatrix} 4 & 0 \\ 0 & 3 \end{pmatrix} & \begin{pmatrix} 0 & 0 \\ 0 & 0 \end{pmatrix} \\ \vspace{-2mm} \\
    \begin{pmatrix} 0 & 0 \\ 0 & 0 \end{pmatrix} & \begin{pmatrix} 2 & 0 \\ 0 & 1 \end{pmatrix}
\end{pmatrix}.
\]
Here $\Pi(U') \in \R^{4\times 4}$ is the permutation matrix that sorts the $4 \times 4$ real matrix $\Pi(A) = \mathrm{diag}(4,2,3,1)$ into $\Pi(\Sigma') = \mathrm{diag}(4,3,2,1)$.
\end{example}

\subsection{Matrix rank and truncated SVDs}
As in the classical case, the rank of a matrix over $K$ is defined as the least number of outer products that sum to it.
\begin{definition}
    The \textbf{rank} of $A \in K^{m\times n}$, denoted $\mathrm{rank}_K(A)$, is the least integer $r \geq 0$ such that
    \[
      A=\sum_{i=1}^{r}u_iv_i^H\qquad\text{for some } u_i\in K^m,\ v_i\in K^n .
    \]
   Equivalently, $\mathrm{rank}_K(A)$ is the least $r \geq 0$ such that $A=PQ$ with $P\in K^{m\times r}$ and
$Q\in K^{r\times n}$.
\end{definition}  
Every $A \in K^{m\times n}$ is a sum of $n$ outer products, $A=\sum_{j=1}^{n}a_je_j^H$, where $a_j$ is the $j$-th column of $A$ and $e_j$ the $j$-th standard basis vector; expanding along the rows gives $m$ terms. Hence $0\le\mathrm{rank}_K(A)\le\min(m,n)$, and $\mathrm{rank}_K(A)=0$ if and only if $A=0$.

\begin{definition}
Let $\pi$ be a structure isomorphism of $K$. The \textbf{multi-rank} of $A \in K^{m\times n}$ is the tuple $(r_1, \ldots, r_\ell)$, where $r_\rho$ is the rank of $\widehat A_\rho$ over $D_\rho$.
\end{definition}

The next result computes the rank of a matrix from its multi-rank and relates it to the number of nonzero singular values in its SVDs.
\begin{proposition}\label{prop:rank}
Let $\pi$ be a structure isomorphism of $K$, and let $(r_1, \ldots, r_\ell)$ be the multi-rank of $A \in K^{m\times n}$. Then
\begin{align}\label{eq:rank}
    \mathrm{rank}_K(A) &= \max_{1 \leq \rho \leq \ell} \left\lceil \frac{r_\rho}{d_\rho} \right\rceil,
\end{align}
and this number equals both of the following:
\begin{enumerate}

\item The number of nonzero diagonal entries of $\Sigma$ in any ordered SVD
  $A=U\Sigma V^H$.
\item The minimum, over all SVDs $A=U\Sigma V^H$, of the number of nonzero diagonal entries of $\Sigma$.
\end{enumerate}
\end{proposition}

We prove Proposition~\ref{prop:rank} in Appendix~\ref{app: proof for rank}. It recovers known notions of rank. For $K = F$, it is the classical statement that the rank is the number of nonzero singular values. For $K = \H$, $\mathrm{rank}_{\H}(A)$ is the rank of $A$ as a quaternion matrix, which equals its number of nonzero singular values~\cite[Thm.~7.3]{zhang1997}. 
For the $\star_M$-product, the multi-rank of $A$ is the multi-rank of~\cite[Def.~3.5]{kilmer2021}, and since every $d_\rho = 1$, Eq.~\eqref{eq:rank} gives $\mathrm{rank}_K(A) = \max_\rho r_\rho$, which is the t-rank~\cite[Def.~3.4]{kilmer2021}.
When some $d_\rho > 1$, the ceiling in Eq.~\eqref{eq:rank} accounts for the fact that each column over $K$ contributes $d_\rho$ columns to $\widehat A_\rho$.

The following example compares three SVDs of the same matrix and relates their numbers of nonzero singular values to its rank.

\begin{example}
    Take $K = \M_2(\R)$ and let $\pi \colon K \to \M_2(\R)$ be the identity.
    Let $A \in K^{3\times 3}$ be the diagonal matrix 
    \begin{equation*}
    A = \begin{pmatrix}
        \begin{pmatrix} 4 & 0\\ 0 & 2 \end{pmatrix} & \begin{pmatrix} 0 & 0\\ 0 & 0 \end{pmatrix} & \begin{pmatrix} 0 & 0\\ 0 & 0 \end{pmatrix} \\ \vspace{-2mm} \\
        \begin{pmatrix} 0 & 0\\ 0 & 0 \end{pmatrix} & \begin{pmatrix} 3 & 0\\ 0 & 1 \end{pmatrix} & \begin{pmatrix} 0 & 0\\ 0 & 0 \end{pmatrix} \\ \vspace{-2mm} \\
        \begin{pmatrix} 0 & 0\\ 0 & 0 \end{pmatrix} & \begin{pmatrix} 0 & 0\\ 0 & 0 \end{pmatrix} & \begin{pmatrix} 0 & 0\\ 0 & 0 \end{pmatrix}
    \end{pmatrix}.
\end{equation*}
    Then $\widehat{A} = \mathrm{diag}(4, 2, 3, 1, 0, 0)$ has rank $4$, so
    $\mathrm{rank}_K(A) = \lceil 4/2 \rceil = 2$. We give three SVDs of $A$ below, where $E_{ij} \in \M_2(\R)$ denotes the matrix with $1$ in position $(i,j)$ and $0$ elsewhere.

    First, $U = V = I_3$ and $\Sigma = A$ is an SVD of $A$ with
    $2 = \mathrm{rank}_K(A)$ nonzero singular values. It is not ordered, since
    $\sigma_{\min}(\pi(\Sigma_{11})) = 2 < 3 = \sigma_{\max}(\pi(\Sigma_{22}))$. We thus see that a non-ordered SVD can also attain the minimum in Proposition~\ref{prop:rank}.

    Second,
    \begin{equation*}
        U = V = \begin{pmatrix}
            E_{11} & 0 & E_{21}\\
            0 & I_2 & 0\\
            E_{12} & 0 & E_{22}
        \end{pmatrix}
    \end{equation*}
    and $\Sigma$ diagonal with
    \begin{equation*}
        \Sigma_{11} = \begin{pmatrix} 4 & 0\\ 0 & 0 \end{pmatrix}, \quad
        \Sigma_{22} = \begin{pmatrix} 3 & 0\\ 0 & 1 \end{pmatrix}, \quad
        \Sigma_{33} = \begin{pmatrix} 2 & 0\\ 0 & 0 \end{pmatrix}
    \end{equation*}
    is an SVD of $A$ with $3 > 2= \mathrm{rank}_K(A)$ nonzero singular values, so it is not ordered, by Proposition~\ref{prop:rank}.

    Finally,
    \begin{equation*}
        U = V = \begin{pmatrix}
            E_{11} & E_{21} & 0\\
            E_{12} & E_{22} & 0\\
            0 & 0 & I_2
        \end{pmatrix}
    \end{equation*}
    and $\Sigma$ diagonal with
    \begin{equation*}
        \Sigma_{11} = \begin{pmatrix} 4 & 0\\ 0 & 3 \end{pmatrix}, \quad
        \Sigma_{22} = \begin{pmatrix} 2 & 0\\ 0 & 1 \end{pmatrix}, \quad
        \Sigma_{33} = \begin{pmatrix} 0 & 0\\ 0 & 0 \end{pmatrix}
    \end{equation*}
    is an SVD of $A$ induced by $\pi$, which is ordered, with $2 = \mathrm{rank}_K(A)$ nonzero singular values.
\end{example}

As in the classical case, truncating an SVD keeps its first $k$ terms.
\begin{definition}
Let $A = U \Sigma V^H$ be an SVD of $A \in K^{m\times n}$, and fix $0 \leq k \leq \min(m,n)$. The
\textbf{rank-$k$ truncation} of this SVD is
\begin{equation}
    A_k = \sum_{i=1}^k u_i \Sigma_{ii} v_i^H = U \Sigma^{(k)} V^H,
\end{equation}
where $u_i$ and $v_i$ denote the $i$-th columns of $U$ and $V$, respectively, and
$\Sigma^{(k)} \in K^{m \times n}$ is the diagonal matrix with $\Sigma^{(k)}_{ii} = \Sigma_{ii}$ for
$i \leq k$ and $\Sigma^{(k)}_{ii} = 0$ for $i > k$.
\end{definition}

Given an ordered SVD $A = U\Sigma V^H$ and $1 \leq k \leq \min(m,n)$, the first $k$ columns of $U$ form $U_k \in K^{m\times k}$, the $k$ \textbf{leading left singular vectors} of $A$. Since $U_k^HU_k = I_k$, the matrix $U_kU_k^H$ is a projector, and
\begin{equation*}
    A_k = U_kU_k^HA.
\end{equation*}
Thus truncating an SVD projects the columns of $A$ onto the right $K$-span of the columns of $U_k$. Algorithm~\ref{alg:LLSV} computes $U_k$ for an SVD induced by $\pi$, without forming the whole SVD over $K$.

\begin{algorithm}[h!]
\caption{$\mathrm{LLSV}_K$}\label{alg:LLSV}
\KwIn{$A \in K^{m\times n}$, integer $1\leq k \leq\min(m,n)$}
\KwOut{$U_k \in K^{m\times k}$, the first $k$ columns of $U$ in an ordered SVD $A = U\Sigma V^H$}
\ForAll{$\rho = 1, \dots, \ell$ \textnormal{\textbf{in parallel}}}{
  $\widehat{A}_\rho \gets \Pi_\rho(A)$\;
  $\widehat{U}_{k,\rho} \gets \mathrm{llsv}(\widehat{A}_\rho, kd_\rho)$\;
}
$U_k \gets \Pi^{-1}\!\left(\bigoplus_{\rho=1}^\ell\widehat{U}_{k,\rho}\right)$\;
\Return{$U_k$}
\end{algorithm}

\paragraph{Computational complexity of $\mathrm{LLSV}_K$.}
Write $n_K \colonequals \dim_F K$ and $d_{\max} \colonequals \max_\rho d_\rho$. Algorithm~\ref{alg:LLSV} applies $\pi$ to the $mn$ entries of $A$ and $\pi^{-1}$ to the $mk$ entries of $U_k$. Each application is an $F$-linear map on a space of dimension $n_K$, so it costs $O(n_K^2)$. The algorithm also computes the $kd_\rho$ leading left singular vectors of each component $\widehat A_\rho$. As in~\cite{BallardKolda2025}, we assume that this costs $O(mn\min(m,n))$ for an $m\times n$ matrix. Since $\widehat A_\rho$ is a matrix of size $md_\rho\times nd_\rho$, the $\rho$-th component costs $O(mn\min(m,n)\,d_\rho^3)$. The total cost is
\[
  O\Big(mn\,n_K^2 \;+\; mn\min(m,n)\sum_{\rho=1}^{\ell} d_\rho^{3}\Big)
\]
operations in $F$. Since $\sum_\rho d_\rho^2 \le n_K$, the second term is at most $O(mn\min(m,n)\,d_{\max}\,n_K)$, and it dominates when $\min(m,n)$ is large compared with $n_K$. When $\pi$ is a fast transform, such as the FFT for the t-product, the first term drops to $O(mn\,n_K\log n_K)$. When $k$ is small, Krylov or randomized methods reduce the second term. The $\ell$ component problems are independent and can be computed in parallel.

\section{An Eckart–Young theorem over $K$} \label{sec: eckart-young over K}
In this section, we characterize the inner products on $K$ for which truncating an ordered SVD gives a best low-rank approximation. Throughout, $K$ is a finite-dimensional, very proper, unital $\ast$-algebra over $F$.

So far, the involution and the inner product on $K$ have been unrelated. For complex matrices, the conjugate transpose is the adjoint for the Frobenius inner product, $\langle AB,C\rangle=\langle B,A^HC\rangle$. We ask the same of $K$.
\begin{definition}\label{axiom for K}
   An inner product $\langle \cdot, \cdot \rangle$ on $K$ satisfies the \textbf{adjoint identities} if
   \begin{equation}
       \langle ab, c \rangle = \langle b, a^\ast c \rangle \quad \text{and} \quad \langle ba, c \rangle = \langle b, c a^\ast \rangle
   \end{equation}
    for all $a,b,c \in K$.
\end{definition}

Theorem~\ref{thm: optimality of rank-k SVD approximation} shows that this is exactly the condition needed for optimal truncations.
\begin{theorem}[Necessary and sufficient condition for optimal truncations]\label{thm: optimality of rank-k SVD approximation}
    Let $K$ be a finite-dimensional, very proper, unital $\ast$-algebra  with an inner product $\langle \cdot, \cdot \rangle$, and let $\|\cdot\|$ denote the induced Frobenius norm on matrices over $K$. Then $\langle \cdot, \cdot \rangle$ satisfies the adjoint identities if and only if, for every $A \in K^{m\times n}$, every ordered SVD of $A$, and every $k$, the rank-$k$ truncation $A_k$ satisfies
    \begin{equation}
        \|A - A_k\| \leq \|A - B\|
    \end{equation}
    for every $B \in K^{m\times n}$ with $\mathrm{rank}_K(B) \leq k$.
\end{theorem}
\begin{proof}
Assume that $\langle\cdot,\cdot\rangle$ satisfies the adjoint identities. Then there exist a structure isomorphism $\pi$ of $K$ and constants $c_\rho > 0$ such that
\begin{equation}\label{eq: relation of inner products}
    \langle a,b \rangle = \sum_{\rho=1}^\ell c_\rho \langle \pi_\rho(a), \pi_\rho(b) \rangle \qquad \text{for all } a,b\in K,
\end{equation}
where $\langle X,Y\rangle = \mathrm{tr}(XY^H)$ on $\mathrm{M}_{d_\rho}(D_\rho)$ for $F = \C$ and $\langle X,Y\rangle = \mathrm{Re}\mathrm{tr}(XY^H)$ on $\mathrm{M}_{d_\rho}(D_\rho)$ for $F = R$. For a proof, see~\cite{ambrose1945} for $F=\C$ and~\cite{ingelstam1966} for $F=\R$ (see also~\cite{balachandran1986, cabrera1989}). Applying Eq.~\eqref{eq: relation of inner products} entrywise gives, for $A,B\in K^{m\times n}$,
\begin{equation}
    \langle A, B \rangle
    = \sum_{i=1}^m \sum_{j=1}^n \langle A_{ij}, B_{ij} \rangle
    = \sum_{\rho=1}^\ell c_\rho \langle \widehat{A}_\rho , \widehat{B}_\rho \rangle,
    \qquad\text{so}\qquad
    \|A\|^2 = \sum_{\rho=1}^\ell c_\rho \|\widehat A_\rho\|^2 .
\end{equation}
Let $A=U\Sigma V^H$ be an ordered SVD, and let $B\in K^{m\times n}$ with $\mathrm{rank}_K(B)\le k$. Conjugating the singular values by unitary elements does not change $A_k$, so by Proposition~\ref{prop: characterization of ordered SVDs} we may assume that the SVD is induced by $\pi$. Then $\Pi_\rho(A_k)$ is a truncation of $\widehat A_\rho$ to its $kd_\rho$ leading singular values, while $\mathrm{rank}(\widehat B_\rho)\le kd_\rho$ by Proposition~\ref{prop:rank}. The Eckart--Young theorem over $D_\rho$\footnote{See~\cite{jia2019, liu2021} for the Eckart--Young result over $\H$.} gives $\|\widehat A_\rho-\Pi_\rho(A_k)\|\le\|\widehat A_\rho-\widehat B_\rho\|$ for every $\rho$, and hence
\begin{equation}
    \|A-A_k\|^2 = \sum_{\rho=1}^\ell c_\rho\|\widehat A_\rho-\Pi_\rho(A_k)\|^2 \le \sum_{\rho=1}^\ell c_\rho\|\widehat A_\rho-\widehat B_\rho\|^2 = \|A-B\|^2 .
\end{equation}

Conversely, assume that truncations of ordered SVDs are optimal, and let $\pi$ be a structure isomorphism of $K$. We prove the adjoint identities in three steps: we show that the inner product is invariant under multiplication by unitary elements on either side, write self-adjoint elements as real combinations of self-adjoint unitary elements, and combine these two results.

\smallskip\noindent\underline{Step (1)}: Fix $a\in K$, a unitary $u\in K$, and a real number $\lambda > \max_\rho \sigma_{\max}(\pi_\rho(a))$. Let
\begin{equation}
    U = \frac{1}{\sqrt{2}}\begin{pmatrix} u & u \\ 1 & -1 \end{pmatrix}, \qquad
    X = \sqrt{2}\, U \begin{pmatrix} \lambda & 0 \\ 0 & a \end{pmatrix}
      = \begin{pmatrix} \lambda u & ua \\ \lambda & -a \end{pmatrix}.
\end{equation}
Then $U$ is unitary, since
\begin{equation}
    U^H U = \frac{1}{2} \begin{pmatrix} u^\ast u + 1 & u^\ast u - 1 \\ u^\ast u - 1 & u^\ast u + 1 \end{pmatrix} = \begin{pmatrix} 1 & 0 \\ 0 & 1 \end{pmatrix}.
\end{equation}
For each $\rho$, $\widehat X_\rho = \sqrt2\,\widehat U_\rho\,\mathrm{diag}(\lambda I_{d_\rho}, \pi_\rho(a))$ with $\widehat U_\rho$ unitary, so the singular values of $\widehat X_\rho$ are $\sqrt2\lambda$, repeated $d_\rho$ times, and $\sqrt2\,\sigma_i(\pi_\rho(a))$, and the former are strictly larger. As in the first part of the proof, for any ordered SVD of $X$, $\Pi_\rho(X_1)$ is a truncation of $\widehat X_\rho$ to its $d_\rho$ leading singular values, which is unique because of this gap. Hence
\begin{equation}
    X_1 = \sqrt{2}\,U\begin{pmatrix} \lambda & 0\\ 0 & 0 \end{pmatrix} = \begin{pmatrix} \lambda u & 0\\ \lambda & 0 \end{pmatrix}.
\end{equation}
For $\tau\in\R$, the outer product
\begin{equation}
    Y(\tau) = \begin{pmatrix} u \\ 1 \end{pmatrix} \begin{pmatrix} \lambda & \tau a \end{pmatrix}
    = \begin{pmatrix} \lambda u & \tau ua \\ \lambda & \tau a \end{pmatrix}
\end{equation}
has rank at most one, and $Y(0)=X_1$. Since
$X - Y(\tau) = \begin{pmatrix} 0 & (1-\tau)ua \\ 0 & -(1+\tau)a \end{pmatrix}$, we have that
\begin{equation}
    f(\tau) = \|X - Y(\tau)\|^2 = (1+\tau^2)\big(\|ua\|^2 + \|a\|^2\big) + 2\tau\big(\|a\|^2 - \|ua\|^2\big).
\end{equation}
By the optimality of $X_1$, $f$ is minimized at $\tau=0$, so its linear coefficient vanishes and $\|ua\|=\|a\|$. The same argument with a unitary $v\in K$ and
\begin{equation}
    W = \frac{1}{\sqrt{2}}\begin{pmatrix} v & 1\\ v & -1 \end{pmatrix}, \quad
    X = \sqrt{2}\begin{pmatrix} \lambda & 0\\ 0 & a \end{pmatrix} W = \begin{pmatrix} \lambda v & \lambda\\ av & -a \end{pmatrix}, \quad
    X_1 = \begin{pmatrix} \lambda v & \lambda\\ 0 & 0 \end{pmatrix}, \quad
    Y(\tau) = \begin{pmatrix} \lambda \\ \tau a \end{pmatrix} \begin{pmatrix} v & 1 \end{pmatrix}
\end{equation}
gives $\|av\|=\|a\|$. Hence $\|uav\|=\|a\|$ for all unitary $u,v\in K$. Since $a\mapsto uav$ is $F$-linear, the polarization identity gives
\begin{equation}\label{eq: unitary invariance}
    \langle uav, ubv \rangle = \langle a, b \rangle \qquad\text{for all } a,b\in K \text{ and unitary } u,v\in K.
\end{equation}

\smallskip\noindent\underline{Step (2)}: Let $s\in K$ be self-adjoint. Each $\pi_\rho(s)$ is self-adjoint, so by the spectral theorem~\cite{zhang1997} $\pi(s)=Q\Lambda Q^H$ with $Q=\bigoplus_\rho Q_\rho$ unitary and $\Lambda=\bigoplus_\rho\Lambda_\rho$ real diagonal. The sign matrices $\mathrm{diag}(\pm1,\dots,\pm1)$ of size $\sum_\rho d_\rho$ span the real diagonal matrices, so $\Lambda=\sum_j\alpha_j\Delta_j$ with $\alpha_j\in\R$ and sign matrices $\Delta_j$. Each $Q\Delta_jQ^H$ is block diagonal, self-adjoint and unitary, so $w_j\colonequals\pi^{-1}(Q\Delta_jQ^H)$ is a self-adjoint unitary element of $K$, and
\begin{equation}\label{eq: s as linear combination of symmetries}
    s=\sum_j\alpha_j w_j .
\end{equation}

\smallskip\noindent\underline{Step (3)}: Taking $v=1$ or $u=1$ in Eq.~\eqref{eq: unitary invariance} gives, for unitary $u,v\in K$,
\begin{equation}
    \langle ua, b\rangle = \langle ua, uu^\ast b\rangle = \langle a, u^\ast b\rangle
    \quad\text{and}\quad
    \langle av, b\rangle = \langle a, bv^\ast\rangle .
\end{equation}
By Eq.~\eqref{eq: s as linear combination of symmetries}, these extend to self-adjoint $s\in K$:
\begin{equation}
    \langle sa, b\rangle = \sum_j\alpha_j\langle w_ja, b\rangle = \sum_j\alpha_j\langle a, w_jb\rangle = \langle a, sb\rangle,
\end{equation}
and likewise $\langle as, b\rangle = \langle a, bs\rangle$. Now let $a\in K$, and write $a=usv^\ast$ with $u,v\in K$ unitary and $s\in K^+$, an SVD of the $1\times1$ matrix $(a)$. Then
\begin{align}
\begin{split}
     \langle ab, c\rangle &= \langle usv^\ast b, c\rangle = \langle sv^\ast b, u^\ast c\rangle = \langle v^\ast b, su^\ast c\rangle = \langle b, vsu^\ast c\rangle = \langle b, a^\ast c\rangle, \text{ and}\\
    \langle ba, c\rangle &= \langle busv^\ast, c\rangle = \langle bus, cv\rangle = \langle bu, cvs\rangle = \langle b, cvsu^\ast\rangle = \langle b, ca^\ast\rangle.
\end{split}
\end{align}
\end{proof}

Kilmer and Martin~\cite{kilmer2011} showed that truncated t-SVDs are
optimal, Kilmer et al.~\cite{kilmer2021} and Avron and
Mor~\cite{avron2025} extended this to the $\star_{\mathbf{M}}$-product when
$\mathbf{M}$ is a nonzero multiple of a unitary matrix, and
Mor~\cite{mor2025} characterized the transforms for which this holds.
Theorem~\ref{thm: optimality of rank-k SVD approximation} recovers this
characterization.

\begin{corollary}[Mor~\cite{mor2025}]\label{cor:mor}
Let $K$ be the  algebra of real tubes under the $\star_{\mathbf{M}}$-product
(Example~\ref{ex: tubal algebras}) and the Euclidean inner product. Then rank-$k$
truncations of $\star_{\mathbf{M}}$-SVDs are optimal in the
Frobenius norm if and only if $\mathbf{M}=\mathbf{D}\mathbf{Q}$, with
$\mathbf{D}$ a nonsingular real diagonal matrix and $\mathbf{Q}$ unitary.
\end{corollary}

\begin{proof}
By Theorem~\ref{thm: optimality of rank-k SVD approximation}, it suffices
to show that the adjoint identities hold if and only if
$\mathbf{M}\mathbf{M}^H$ is diagonal. Let
$\mathbf{G}=(\mathbf{M}\mathbf{M}^H)^{-1}$. For real tubes,
$\langle\boldsymbol{a},\boldsymbol{b}\rangle
=\widehat{\boldsymbol{b}}^H\mathbf{G}\,\widehat{\boldsymbol{a}}$. Since
$\widehat{\boldsymbol{a}\star_{\mathbf{M}}\boldsymbol{b}}
=\widehat{\boldsymbol{a}}\bullet\widehat{\boldsymbol{b}}$ and
$\widehat{\boldsymbol{a}^*}=\overline{\widehat{\boldsymbol{a}}}$, we have
\[
  \langle\boldsymbol{a}\star_{\mathbf{M}}\boldsymbol{b},\boldsymbol{c}\rangle
  =\widehat{\boldsymbol{c}}^H\mathbf{G}\,
   \big(\widehat{\boldsymbol{a}}\bullet\widehat{\boldsymbol{b}}\big)
  \qquad\text{and}\qquad
  \langle\boldsymbol{b},\boldsymbol{a}^*\star_{\mathbf{M}}\boldsymbol{c}\rangle
  =\widehat{\boldsymbol{c}}^H
   \big(\widehat{\boldsymbol{a}}\bullet\mathbf{G}\,\widehat{\boldsymbol{b}}\big).
\]
The vectors $\widehat{\boldsymbol{b}}$ and $\widehat{\boldsymbol{c}}$ span
$\mathbb{C}^n$, so the first identity holds if and only if
$\mathbf{G}\,(\widehat{\boldsymbol{a}}\bullet\boldsymbol{x})
=\widehat{\boldsymbol{a}}\bullet\mathbf{G}\boldsymbol{x}$ for every
$\boldsymbol{a}$ and every $\boldsymbol{x}\in\mathbb{C}^n$, that is, if and
only if $G_{jl}(\widehat{a}_l-\widehat{a}_j)=0$ for all $j,l$. Since some
$\widehat{\boldsymbol{a}}$ has distinct entries, this holds if and only if
$\mathbf{G}$, and hence $\mathbf{M}\mathbf{M}^H$, is diagonal. The second
identity is the same as the first, since $K$ is commutative. Finally,
$\mathbf{M}\mathbf{M}^H$ is diagonal if and only if
$\mathbf{M}\mathbf{M}^H=\mathbf{D}^2$ with $\mathbf{D}$ positive diagonal,
that is, if and only if $\mathbf{Q}=\mathbf{D}^{-1}\mathbf{M}$ is unitary.
\end{proof}

For the $\star_G$-product, Hoyos et al.~\cite{hoyos2026} proved that
truncations of the $\star_G$-SVD are optimal using the Plancherel formula.
In our framework, this follows from the adjoint identities.

\begin{corollary}[Hoyos et al.~\cite{hoyos2026}]\label{cor:starG}
Let $K=F[G]$ be the group algebra of a finite group $G$, with the product,
involution and inner product of Example~\ref{ex: group algebras}. Then
rank-$k$ truncations of $\star_G$-SVDs are optimal in the Frobenius norm.
\end{corollary}

\begin{proof}
By Theorem~\ref{thm: optimality of rank-k SVD approximation}, it suffices
to check the adjoint identities. For $a,c\in K$ and $g\in G$,
\[
  (a^*c)(g)=\sum_{h\in G}\overline{a(h^{-1})}\,c(h^{-1}g)
  =\sum_{h\in G}\overline{a(h)}\,c(hg).
\]
Hence, for $a,b,c\in K$, substituting $k=hg$,
\[
  \langle ab,c\rangle
  =\sum_{h,k\in G}a(h)\,b(h^{-1}k)\,\overline{c(k)}
  =\sum_{g,h\in G}a(h)\,b(g)\,\overline{c(hg)}
  =\sum_{g\in G}b(g)\,\overline{(a^*c)(g)}
  =\langle b,a^*c\rangle.
\]
The identity $\langle ba,c\rangle=\langle b,ca^*\rangle$ follows in the same
way, substituting $k=gh$.
\end{proof}

Miao and Kou~\cite[Thm.~2]{miaokou2021tsvd} showed that truncated TQt-SVDs are
optimal when the transform is a nonzero real multiple of a unitary
quaternion matrix. As for the $\star_{\mathbf{M}}$-product, the class of
transforms for which this holds is larger.

\begin{corollary}[Miao and Kou~\cite{miaokou2021tsvd}]\label{cor:qt}
Let $K=\mathbb{H}^n$ be the algebra of quaternion tubes under the
$\star_{QT}$-product with an invertible quaternion matrix $\mathbf{T}$, and
the inner product
$\langle\boldsymbol{a},\boldsymbol{b}\rangle
=\operatorname{Re}(\boldsymbol{b}^H\boldsymbol{a})$. Then rank-$k$
truncations of TQt-SVDs of third-order quaternion tensors are optimal in
the Frobenius norm if and only if $\mathbf{T}=\mathbf{D}\mathbf{Q}$, with
$\mathbf{D}$ a nonsingular real diagonal matrix and $\mathbf{Q}$ a unitary
quaternion matrix.
\end{corollary}

\begin{proof}
As in the proof of Corollary~\ref{cor:mor}, let
$\mathbf{G}=(\mathbf{T}\mathbf{T}^H)^{-1}$, so that
$\langle\boldsymbol{a},\boldsymbol{b}\rangle
=\operatorname{Re}(\widehat{\boldsymbol{b}}^H\mathbf{G}\,
\widehat{\boldsymbol{a}})$. Then
\[
  \langle\boldsymbol{a}\star_{QT}\boldsymbol{b},\boldsymbol{c}\rangle
  -\langle\boldsymbol{b},\boldsymbol{a}^*\star_{QT}\boldsymbol{c}\rangle
  =\operatorname{Re}\sum_{j,l=1}^{n}\overline{\widehat{c}_j}\,
   \big(G_{jl}\,\widehat{a}_l-\widehat{a}_j\,G_{jl}\big)\,\widehat{b}_l .
\]
Since $\widehat{\boldsymbol{b}}$ and $\widehat{\boldsymbol{c}}$ range over
$\mathbb{H}^n$, the first identity holds if and only if
$G_{jl}\,\widehat{a}_l=\widehat{a}_j\,G_{jl}$ for all $j,l$ and every
$\widehat{\boldsymbol{a}}\in\mathbb{H}^n$. Taking $\widehat{a}_l=0$ and
$\widehat{a}_j=1$ gives $G_{jl}=0$ for $j\ne l$, and $G_{jj}$ must commute
with every quaternion, so it is real. Hence the first identity holds if
and only if $\mathbf{G}$, and therefore $\mathbf{T}\mathbf{T}^H$, is a real
diagonal matrix. In that case the second identity also holds, by the same
computation and $\operatorname{Re}(xy)=\operatorname{Re}(yx)$. Finally,
$\mathbf{T}\mathbf{T}^H$ is real diagonal if and only if
$\mathbf{T}\mathbf{T}^H=\mathbf{D}^2$ with $\mathbf{D}$ positive diagonal,
that is, if and only if $\mathbf{Q}=\mathbf{D}^{-1}\mathbf{T}$ is unitary.
\end{proof}

Yu et al.~\cite{yu2024} introduced an SVD for tensors of reduced biquaternions and used its truncations for color video compression, without an optimality guarantee. 
Theorem~\ref{thm: optimality of rank-k SVD approximation} provides one.

\begin{corollary}\label{cor:rb}
For third-order reduced biquaternion tensors with the $\ast_{\mathrm{Ht}}$-product of Yu et al.~\cite{yu2024}, rank-$k$ truncations of Ht-SVDs are optimal in the Frobenius norm.
\end{corollary}

\begin{proof}
As in the proof of Corollary~\ref{cor:existence}, the splitting $q\mapsto(q_{(c),1},q_{(c),2})$ followed by the discrete Fourier transform
along the tubes is a $*$-isomorphism $\pi$ onto $\mathbb{C}^{2n}$. 
Since $\|q\|^2=\tfrac12\big(|q_{(c),1}|^2+|q_{(c),2}|^2\big)$, Parseval's identity gives
$\langle a,b\rangle=\frac{1}{2n}\sum_\rho \operatorname{Re}\big(\pi_\rho(a)\overline{\pi_\rho(b)}\big)$. 
Hence
\[
  \langle ab,c\rangle
  =\frac{1}{2n}\sum_\rho\operatorname{Re}
   \big(\pi_\rho(b)\,\overline{\overline{\pi_\rho(a)}\,\pi_\rho(c)}\big)
  =\langle b,a^*c\rangle,
\]
and the second identity is the same, since the algebra is commutative. The result follows from Theorem~\ref{thm: optimality of rank-k SVD approximation}.
\end{proof}

\section{SVD-Based Tensor Decompositions}\label{sec:tensor}
In this section, we extend the Tucker and tensor train (TT) decompositions to tensors with entries in $K$. Using the SVD of matrices over $K$, we generalize the HOSVD and ST-HOSVD algorithms for the Tucker decomposition and the TT-SVD algorithm for the TT decomposition. We show that the TT-SVD is quasi-optimal whenever the inner product satisfies the adjoint identities, while the HOSVD and ST-HOSVD are quasi-optimal for commutative $K$. We give an example where the HOSVD guarantee fails for noncommutative $K$. 

All algorithms in this section call $\mathrm{LLSV}_K$ (Algorithm~\ref{alg:LLSV}), either with a target number of columns or with a relative error tolerance. The following remark explains how a tolerance determines the number of columns in the output of $\mathrm{LLSV}_K$. 
\begin{remark}[Choosing the sizes from a tolerance]\label{rem: LLSV tolerance}
   Let $A\in K^{m\times n}$, and assume that the inner product on $K$ satisfies the adjoint identities, with weights $c_\rho$ as in Eq.~\eqref{eq: relation of inner products}. Let $\sigma_{\rho,1}\ge\sigma_{\rho,2}\ge\cdots$ be the singular values of $\widehat A_\rho$. Each column over $K$ accounts for $d_\rho$ of them in component $\rho$, so we assign to its $s$-th column the energy
\[
  e_{\rho,s} = c_\rho\sum_{t=(s-1)d_\rho+1}^{sd_\rho}\sigma_{\rho,t}^2 .
\]
Given $\delta\ge0$, $\mathrm{LLSV}_K(A,\delta)$ discards the columns of smallest energy, across all components, while their total energy is at most $\delta^2$. If $k_\rho$ columns remain in component $\rho$, it runs Algorithm~\ref{alg:LLSV} with $\mathrm{llsv}(\widehat A_\rho,k_\rho d_\rho)$ in component $\rho$ and returns $U\in K^{m\times k}$ with $k=\max_\rho k_\rho$, padding component $\rho$ with zero columns. Then $\|A-UU^HA\|^2$ equals the discarded energy, which is at most $\delta^2$.
\end{remark}

We now turn to tensors. The \textbf{mode-$k$ unfolding} of a tensor $\calA \in \A^{n_1\times\cdots\times n_d}$ is the matrix $\calA_{(k)} \in \A^{n_k\times N_k}$, where $N_k = \prod_{q\neq k} n_q$, defined by
   \[
     \calA_{(k)}(i_k,\beta_k) = \calA(i_1,\ldots,i_d),
     \quad
     \beta_k = 1 + \sum_{q\neq k}
       \Bigl(\prod_{p=1,\, p\neq k}^{q-1} n_p\Bigr)
       (i_{q}-1).
   \]
This is the unfolding of~\cite{BallardKolda2025}. Since unfolding only rearranges the entries, $\|\calA_{(k)}\| = \|\calA\|$ for every $k$, with the norms of Eqs.~\eqref{eq: inner product of matrix} and~\eqref{eq: inner product of tensor}.

\subsection{Tucker decomposition}
A Tucker decomposition writes a tensor as a core tensor multiplied by a factor matrix in each mode. We first define these products over $K$.

The \textbf{mode-$k$ tensor-times-matrix (TTM) product} of $\calA \in K^{n_1\times\cdots\times n_d}$ with $U \in K^{m \times n_k}$ is the tensor $\calA \times_k U \in K^{n_1\times\cdots\times n_{k-1}\times m \times n_{k+1}\times\cdots\times n_d}$ defined by $(\calA\times_k U)_{(k)} = U\calA_{(k)}$, that is,
\[
  (\calA \times_k U)(i_1, \dots, i_{k-1}, j, i_{k+1}, \dots, i_d)
  = \sum_{i_k=1}^{n_k} U(j, i_k)\, \calA(i_1, \dots, i_d).
\]
By associativity of matrix multiplication, $(\calA \times_k U) \times_k V = \calA \times_k (VU)$. Since $U$ multiplies the entries of $\calA$ from the left, products in different modes need not commute when $K$ is noncommutative: in general, $(\calA \times_k U) \times_q V \neq (\calA \times_q V) \times_k U$ for $k \neq q$ (see Example~\ref{ex: HOSVD fails}). We therefore 
set the convention
\[
  \bigl(\calA \times_1 U_1 \times_2 \cdots \times_d U_d\bigr)(i_1, \dots, i_d)
  = \sum_{j_1, \dots, j_d} U_1(i_1, j_1) \cdots U_d(i_d, j_d)\, \calA(j_1, \dots, j_d).
\]
When $K$ is commutative, products in different modes commute and the order does not matter, as in the real case. With this convention, the Tucker decomposition is defined as in the classical case.
\begin{definition}
A \textbf{Tucker decomposition} of $\calA \in K^{n_1 \times \cdots \times n_d}$ with \textbf{core size} $(r_1, \dots, r_d)$ is a factorization
\begin{equation}\label{eq:tucker}
    \calA = \calC \times_1 U_1 \times_2 \cdots \times_d U_d
\end{equation}
with core $\calC \in K^{r_1 \times \cdots \times r_d}$ and factor matrices $U_k \in K^{n_k \times r_k}$.
\end{definition}

We write $\calA = \llbracket \calC ; U_1, \dots, U_d \rrbracket$, and call this representation, stored as the core and the factor matrices, a \textbf{Tucker tensor}. See Figure~\ref{fig:tucker}.

\begin{figure}[h]
\vspace{0.5\baselineskip}
\centering
\setlength{\abovecaptionskip}{4pt}
\begin{tikzpicture}[
  core/.style={circle, draw, thick, minimum size=0.85cm, inner sep=1pt},
  line/.style={thick}
]
  \node[core] (c)    at (2.4,1.15) {$\mathcal{C}$};
  \node[core] (u1)   at (0,0)      {$U_1$};
  \node[core] (u2)   at (1.6,0)    {$U_2$};
  \node       (dots) at (3.2,0)    {$\cdots$};
  \node[core] (ud)   at (4.8,0)    {$U_d$};

  \draw[line] (c) -- node[above left]  {$r_1$} (u1);
   \draw[line] (c) -- node[pos=0.65, right] {$r_2$} (u2);
  \draw[line] (c) -- node[above right] {$r_d$} (ud);

  \foreach \nd/\idx in {u1/1, u2/2, ud/d}
    \draw[line] (\nd.south) -- ++(0,-0.6)
      node[midway, right] {$n_{\idx}$};
\end{tikzpicture}
\caption{The Tucker decomposition~\eqref{eq:tucker} as a tensor network.
Each edge is labeled by its size.}
\label{fig:tucker}
\end{figure}

Algorithm~\ref{alg:HOSVD} generalizes the higher-order SVD (HOSVD) to tensors over $K$. It computes each factor matrix from the leading left singular vectors of an unfolding, then projects $\calA$ onto the factors to form the core. As in the classical case, it can be run with a target core size or with a relative error tolerance $\varepsilon$.

\begin{algorithm}
\caption{HOSVD$_K$}\label{alg:HOSVD}
\KwIn{$\calA \in K^{n_1 \times \cdots \times n_d}$; target core size $(r_1, \dots, r_d)$ or
  relative error tolerance $\varepsilon \ge 0$}
\KwOut{Tucker tensor $\calT = \llbracket \calC; U_1, \dots, U_d \rrbracket$ with core size
  $(r_1, \dots, r_d)$ such that $\calT \approx \calA$}
\ForAll{$k = 1, \dots, d$}{
  $U_k \gets \mathrm{LLSV}_K(\calA_{(k)},\, r_k \text{ or } \varepsilon\|\calA\|/\sqrt{d})$
}
$\calC \gets \calA \times_d U_d^H \times_{d-1} U_{d-1}^H \cdots \times_1 U_1^H$
  \tcp*{modes in reverse order}
\Return{$\calT = \llbracket \calC; U_1, \dots, U_d \rrbracket$}
\end{algorithm}

\paragraph{Computational complexity of $\mathrm{HOSVD}_K$.}
Let $N=\prod_{q=1}^d n_q$, $N_k=N/n_k$, $n_K=\dim_F K$ and $\kappa_K=\sum_{\rho=1}^\ell d_\rho^3$. Each call to $\mathrm{LLSV}_K$ applies $\Pi$ to its input and $\Pi^{-1}$ to its output. Since $\Pi$ respects products and conjugate transposes, the same result is obtained by applying $\Pi$ to $\calA$ once, at cost $O(Nn_K^2)$, carrying out all LLSVs and mode products on the components, and applying $\Pi^{-1}$ only to the core and the factors. We count this implementation, which is the one we use in our experiments. By the complexity analysis of Algorithm~\ref{alg:LLSV}, the LLSV of $\calA_{(k)}\in K^{n_k\times N_k}$ costs $O(N\min(n_k,N_k)\,\kappa_K)$ on the components. A mode-$k$ product $\calX\times_k U$ with $U\in K^{r\times n_k}$ costs $O(r\,|\calX|\,\kappa_K)$, where $|\calX|$ is the number of entries of $\calX$, since in component $\rho$ each multiplication in $K$ becomes a product of $d_\rho\times d_\rho$ matrices. Forming the core therefore costs $O(N\kappa_K\sum_k r_k)$, which is dominated by the LLSVs since $r_k\le\min(n_k,N_k)$. The total cost is
\[
  O\Bigl(Nn_K^2 + N\kappa_K\sum_{k=1}^d\min(n_k,N_k)\Bigr),
\]
which for an $n\times\cdots\times n$ tensor is $O(n^dn_K^2 + d\,n^{d+1}\kappa_K)$, compared with $O(d\,n^{d+1})$ for $K=F$.\\

As for real tensors, sequential truncation reduces the cost of HOSVD~\cite{Vannieuwenhoven2012,hackbusch2019}. Algorithm~\ref{alg:STHOSVD} generalizes the sequentially truncated HOSVD (ST-HOSVD) to tensors over $K$. It compresses the tensor in each mode as soon as the factor matrix for that mode is computed, so later calls to $\mathrm{LLSV}_K$ act on smaller tensors. 
\begin{algorithm}[h!]
\caption{ST-HOSVD$_K$}\label{alg:STHOSVD}
\KwIn{$\calA \in K^{n_1 \times \cdots \times n_d}$; target core size $(r_1, \dots, r_d)$ or
  relative error tolerance $\varepsilon \ge 0$}
\KwOut{Tucker tensor $\calT = \llbracket \calC; U_1, \dots, U_d \rrbracket$ with core size
  $(r_1, \dots, r_d)$ such that $\calT \approx \calA$}
$\calC \gets \calA$\;
\For{$k = 1, \dots, d$}{
  $U_k \gets \mathrm{LLSV}_K(\calC_{(k)},\, r_k \text{ or } \varepsilon\|\calA\|/\sqrt{d})$
  \tcp*{from the partially compressed core}
  $\calC \gets \calC \times_k U_k^H$\;
}
\Return{$\calT = \llbracket \calC; U_1, \dots, U_d \rrbracket$}
\end{algorithm}

\paragraph{Computational complexity of $\mathrm{ST\text{-}HOSVD}_K$.}
As for $\mathrm{HOSVD}_K$, we apply $\Pi$ to $\calA$ once and work on the components. When mode $k$ is processed, the current core has size
$r_1\times\cdots\times r_{k-1}\times n_k\times\cdots\times n_d$, so its
mode-$k$ unfolding is $n_k\times M_k$ with
$M_k=\prod_{q<k}r_q\prod_{q>k}n_q$. The LLSV costs $O(n_kM_k\min(n_k,M_k)\,\kappa_K)$, and the compression, $O(r_kn_kM_k\,\kappa_K)$, is no larger. The total cost is
\[
  O\Bigl(Nn_K^2 + \kappa_K\sum_{k=1}^d n_kM_k\min(n_k,M_k)\Bigr),
\]
which for an $n\times\cdots\times n$ tensor with core size $(r,\dots,r)$ and $r\le n/2$ is $O(n^dn_K^2 + n^{d+1}\kappa_K)$, compared with $O(n^{d+1})$ for $K=F$.\\

When $K$ is commutative, HOSVD and ST-HOSVD inherit the classical
quasi-optimality bounds from
Theorem~\ref{thm: optimality of rank-k SVD approximation}. 
\begin{theorem}[Quasi-optimality of $\mathrm{HOSVD}_K$ and $\mathrm{ST\text{-}HOSVD}_K$]
\label{cor: optimality for Tucker decompositions}
Let $K$ be a finite-dimensional, unital, very proper, commutative $\ast$-algebra with an inner product that satisfies the adjoint identities, and let $\calT$ be the Tucker tensor computed from $\calA \in K^{n_1\times\cdots\times n_d}$ by Algorithm~\ref{alg:HOSVD} or~\ref{alg:STHOSVD}. If the algorithm is run with core size $(r_1,\dots,r_d)$, then
\begin{equation}\label{eq:hosvd-bound}
    \|\calA - \calT\| \leq \sqrt{d}\, \|\calA - \calT^\star\|,
\end{equation}
where $\calT^\star$ is a best Tucker approximation of $\calA$ with that core size. 
\end{theorem}

We prove Theorem~\ref{cor: optimality for Tucker decompositions} in Appendix~\ref{app: proof of corollaries}. 
Wang and Yang~\cite{wang2022hot} extended HOSVD and ST-HOSVD to tubal tensors and proved the $\sqrt{d}$ bound when $\mathbf{M}$ is a nonzero multiple of a unitary matrix. 
Theorem~\ref{cor: optimality for Tucker decompositions} gives this bound for any $\star_{\mathbf{M}}$-product with
$\mathbf{M}=\mathbf{D}\mathbf{Q}$.

\begin{corollary}[Wang and Yang~\cite{wang2022hot}]\label{cor:hot-svd}
For tubal tensors under the $\star_{\mathbf{M}}$-product with $\mathbf{M}=\mathbf{D}\mathbf{Q}$, $\mathrm{HOSVD}_K$ and $\mathrm{ST\text{-}HOSVD}_K$ satisfy the bound~\eqref{eq:hosvd-bound}.
\end{corollary}

\begin{proof}
The algebra $K$ is commutative, and by the proof of Corollary~\ref{cor:mor} the Euclidean inner product satisfies the adjoint identities. The bound follows from Theorem~\ref{cor: optimality for Tucker decompositions}.
\end{proof}

For noncommutative $K$, the theorem can fail. Each $U_k$ captures the column space of an unfolding, but the factor matrices of the other modes multiply the entries from the left, which need not preserve that column space. We illustrate this in the following example.
\begin{example}\label{ex: HOSVD fails}
Let $K = \mathrm{M}_2(\mathbb{R})$ with the Frobenius inner product,
and let
\[
R = \begin{pmatrix} 0 & -1 \\ 1 & 0 \end{pmatrix}, \quad
D = \begin{pmatrix} 1 & 0 \\ 0 & -1 \end{pmatrix}, \quad
S = \begin{pmatrix} 0 & 1 \\ 1 & 0 \end{pmatrix}.
\]
The Tucker tensor
$\mathcal{A}(i_1,i_2,i_3) = u_{i_1} v_{i_2} w_{i_3}$ with core $I$
and factors $u = (I, S)^\top$, $v = (2I, D)^\top$, and $w = (2I, R)^\top$
has slices
\[
\mathcal{A}(:,:,1) = \begin{pmatrix} 4I & 2D \\ 4S & 2R \end{pmatrix},
\qquad
\mathcal{A}(:,:,2) = \begin{pmatrix} 2R & -S \\ 2D & -I \end{pmatrix}.
\]
The unfolding $\mathcal{A}_{(1)}$ has rank one, and since $SDS = -D$
and $SRS = -R$, we have
\[
\mathcal{A}_{(2)}\mathcal{A}_{(2)}^H
= \mathcal{A}_{(3)}\mathcal{A}_{(3)}^H
= \begin{pmatrix} 40I & 0 \\ 0 & 10I \end{pmatrix},
\]
so $\mathcal{A}_{(2)}$ and $\mathcal{A}_{(3)}$ have rank two over
$K$, with distinct singular values.
The leading singular vectors are thus unique up to unitary elements,
which do not affect the output, and HOSVD$_K$ with core size
$(1,1,1)$ uses $U_1 = u/\sqrt{2}$ and $U_2 = U_3 = (I, 0)^\top$.
It keeps the fiber $\mathcal{A}(:,1,1)$ and zeroes the rest, so
$\|\mathcal{A} - \mathcal{T}\| = (3/5)\|\mathcal{A}\|$,
although $\mathcal{A}$ has core size $(1,1,1)$. As the singular
values are separated, the failure persists for tensors near
$\mathcal{A}$. ST-HOSVD$_K$ recovers $\mathcal{A}$ exactly, since it
processes mode 1 first.
\end{example}

For quaternion tensors, Ya and Yang~\cite{YaYang2023} give an example in which the error of the HOSVD exceeds the sum of the discarded tails.

\subsection{Tensor train decomposition}

A tensor train (TT) decomposition writes a tensor as a chain of cores, one for each mode, with matrices at the two ends and three-way tensors in between. Since $K$ need not be commutative, we multiply the cores in the order of the modes.

\begin{definition}
\label{def:tt}
A \textbf{tensor train (TT) decomposition} of $\calA \in K^{n_1 \times \cdots \times n_d}$ with \textbf{TT sizes} $(r_1, \dots, r_{d-1})$ is a factorization
\begin{equation}\label{eq:tt}
    \calA(i_1, \dots, i_d) = \sum_{j_1=1}^{r_1} \cdots \sum_{j_{d-1}=1}^{r_{d-1}}
    \calC_1(i_1, j_1)\, \calC_2(j_1, i_2, j_2) \cdots \calC_{d-1}(j_{d-2}, i_{d-1}, j_{d-1})\,
    \calC_d(j_{d-1}, i_d)
\end{equation}
with cores $\calC_1 \in K^{n_1 \times r_1}$, $\calC_k \in K^{r_{k-1} \times n_k \times r_k}$ for $k = 2, \dots, d-1$, and $\calC_d \in K^{r_{d-1} \times n_d}$.
\end{definition}

We write $\calA = \llbracket \calC_1, \dots, \calC_d \rrbracket$, and call this representation, stored as its cores, a \textbf{TT tensor}. See Figure~\ref{fig:tt}.

\begin{figure}[h]
\vspace{0.5\baselineskip}
\centering
\setlength{\abovecaptionskip}{4pt}
\begin{tikzpicture}[
  core/.style={circle, draw, thick, minimum size=0.85cm, inner sep=1pt},
  line/.style={thick}
]
  \node[core] (c1)   at (0,0)   {$\mathcal{C}_1$};
  \node[core] (c2)   at (1.8,0) {$\mathcal{C}_2$};
  \node       (dots) at (3.6,0) {$\cdots$};
  \node[core] (cd1)  at (5.4,0) {$\mathcal{C}_{d-1}$};
  \node[core] (cd)   at (7.2,0) {$\mathcal{C}_d$};

  \draw[line] (c1)   -- node[above] {$r_1$}     (c2);
  \draw[line] (c2)   -- node[above] {$r_2$}     (dots);
  \draw[line] (dots) -- node[above] {$r_{d-2}$} (cd1);
  \draw[line] (cd1)  -- node[above] {$r_{d-1}$} (cd);

  \foreach \nd/\idx in {c1/1, c2/2, cd1/d-1, cd/d}
    \draw[line] (\nd.south) -- ++(0,-0.6)
      node[midway, right] {$n_{\idx}$};
\end{tikzpicture}
\caption{The TT decomposition~\eqref{eq:tt} as a tensor network. Each
edge is labeled by its size.}
\label{fig:tt}
\end{figure}
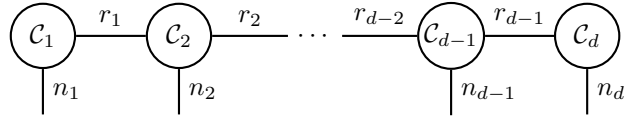

Algorithm~\ref{alg:TT-SVD} generalizes the TT-SVD to tensors over $K$. It computes one core at a time from the leading left singular vectors of a reshaping, where $\mathrm{reshape}$ rearranges the entries in the same index order as the unfoldings. As in the classical case, it can be run with target TT sizes or with a relative error tolerance $\varepsilon$.
\begin{algorithm}[h!]
\caption{TT-SVD$_K$}\label{alg:TT-SVD}
\KwIn{$\calA \in K^{n_1 \times \cdots \times n_d}$; target TT sizes $(r_1, \dots, r_{d-1})$ or
  relative error tolerance $\varepsilon \ge 0$}
\KwOut{cores $\calC_1, \dots, \calC_d$ of a TT approximation of $\calA$ with TT sizes
  $(r_1, \dots, r_{d-1})$}
$r_0 \gets 1$; \quad $Y \gets \calA$\;
\For{$k = 1, \dots, d-1$}{
  $\bar Y \gets \mathrm{reshape}\bigl(Y,\ (r_{k-1} n_k) \times (n_{k+1} \cdots n_d)\bigr)$\;
  $C_k \gets \mathrm{LLSV}_K\bigl(\bar Y,\ r_k \text{ or } \varepsilon\|\calA\|/\sqrt{d-1}\bigr)$
    \tcp*{$C_k \in K^{(r_{k-1} n_k) \times r_k}$}
      $Y \gets C_k^H \bar Y$ \tcp*{$Y \in K^{r_k \times (n_{k+1} \cdots n_d)}$}
  $\calC_k \gets \mathrm{reshape}(C_k,\ r_{k-1} \times n_k \times r_k)$\;

}
$\calC_d \gets \mathrm{reshape}(Y,\ r_{d-1} \times n_d)$\;
\Return{$\calC_1, \dots, \calC_d$}
\end{algorithm}

\paragraph{Computational complexity of $\mathrm{TT\text{-}SVD}_K$.}
As for $\mathrm{HOSVD}_K$, we apply $\Pi$ to $\calA$ once and work on the components. Let $P_k=\prod_{q>k}n_q$ and $r_0=1$. At step $k$, the LLSV of $\bar Y\in K^{(r_{k-1}n_k)\times P_k}$ costs $O(r_{k-1}n_kP_k\min(r_{k-1}n_k,P_k)\,\kappa_K)$, and the product $C_k^H\bar Y$, $O(r_kr_{k-1}n_kP_k\,\kappa_K)$, is no larger since $r_k\le\min(r_{k-1}n_k,P_k)$. The total cost is
\[
  O\Bigl(Nn_K^2 + \kappa_K\sum_{k=1}^{d-1} r_{k-1}n_kP_k\min(r_{k-1}n_k,P_k)\Bigr),
\]
which for an $n\times\cdots\times n$ tensor with all TT sizes equal to $r$ is $O\bigl(n^dn_K^2 + (n^{d+1}+r^2n^d)\,\kappa_K\bigr)$, compared with $O(n^{d+1}+r^2n^d)$ for $K=F$.\\

Unlike TTM products, the products in a TT decomposition follow the order of the modes, so Theorem~\ref{thm: optimality of rank-k SVD approximation} gives the classical quasi-optimality bound even for noncommutative $K$.
\begin{theorem}[Quasi-optimality of $\mathrm{TT\text{-}SVD}_K$]\label{cor: optimality for TT-SVD}
Let $K$ be a finite-dimensional, unital, very proper $\ast$-algebra with an inner product that satisfies the adjoint identities, and let $\calT$ be the TT tensor computed from $\calA\in K^{n_1\times\cdots\times n_d}$ by Algorithm~\ref{alg:TT-SVD}. If the algorithm is run with TT sizes $(r_1,\dots,r_{d-1})$, then
\begin{equation}\label{eq:tt-bound}
    \|\calA - \calT\| \leq \sqrt{d-1}\, \|\calA - \calT^\star\|,
\end{equation}
where $\calT^\star$ is a best TT approximation of $\calA$ with those TT sizes. 
\end{theorem}
We prove Theorem~\ref{cor: optimality for TT-SVD} in Appendix~\ref{app: proof of corollaries}.
For tubal tensors under the t-product, Ahmadi-Asl et al.~\cite{ahmadiasl2026} introduced a tubal tensor train and bounded the error of their TT-SVD-type
algorithm by the discarded tails. Theorem~\ref{cor: optimality for TT-SVD} gives the quasi-optimality bound for any $\star_{\mathbf{M}}$-product with $\mathbf{M}=\mathbf{D}\mathbf{Q}$, and in particular for the t-product.

\begin{corollary}[Ahmadi-Asl et al.~\cite{ahmadiasl2026}]\label{cor:ttt}
For tubal tensors under the t-product, $\mathrm{TT\text{-}SVD}_K$ satisfies
the bound~\eqref{eq:tt-bound}.
\end{corollary}

\begin{proof}
The t-product is the $\star_{\mathbf{M}}$-product with $\mathbf{M}=\mathbf{F}_n=\sqrt{n}\,\mathbf{W}$, where $\mathbf{W}$ is
unitary.
By the proof of Corollary~\ref{cor:mor}, the Euclidean inner product then satisfies the adjoint identities, and the bound follows from
Theorem~\ref{cor: optimality for TT-SVD}.
\end{proof}

Over the quaternions, TT decompositions with cores multiplied in the order of the modes were introduced by Miao et al.~\cite{miao2022}, who
constructed exact decompositions from successive quaternion matrix factorizations but gave no error bound for truncated decompositions.
Theorem~\ref{cor: optimality for TT-SVD} provides one.

\begin{corollary}\label{cor:qtt}
For quaternion tensors, $\mathrm{TT\text{-}SVD}_K$ satisfies the
bound~\eqref{eq:tt-bound}.
\end{corollary}

\begin{proof}
For $K=\mathbb{H}$, the TT decomposition of Definition~\ref{def:tt} is that of~\cite[Def.~1]{miao2022}. Take the inner product
$\langle p,q\rangle=\operatorname{Re}(p\overline{q})$. 
Since $\operatorname{Re}(xy)=\operatorname{Re}(yx)$ for $x,y\in\mathbb{H}$, we have
$\langle pq,r\rangle=\operatorname{Re}(q\,\overline{r}\,p)
=\langle q,\overline{p}\,r\rangle$ 
and
$\langle qp,r\rangle=\operatorname{Re}(q\,p\,\overline{r})
=\langle q,r\,\overline{p}\rangle$,
so the adjoint identities hold. The bound follows from Theorem~\ref{cor: optimality for TT-SVD}.
\end{proof}

\subsection{Tree tensor network decomposition}\label{sec:ttn}

A tree tensor network (TTN) decomposition generalizes the Tucker and TT decompositions by arranging the cores along a rooted tree, rather than a
star (Tucker) or a path (TT).

Let $T$ be a rooted tree with root $v_0$ and vertex set $\mathcal{V}(T)$. 
Each vertex $v\ne v_0$ is joined to its parent by an \textbf{internal edge}, and $T$ has $d$ \textbf{external edges}, labeled $1,\dots,d$, each attached to one vertex.
Every vertex without children has at least one external edge. 
With the external edges $1,\dots,d$ drawn from left to right, the \textbf{order of $T$} lists the vertices from left to right, placing
each vertex after its children. 
For example, the order of the TT tree in Figure~\ref{fig:tt} is $\mathcal{C}_1,\dots,\mathcal{C}_d$, and the order of the Tucker tree in
Figure~\ref{fig:tucker} is $U_1,\dots,U_d,\mathcal{C}$.

\begin{definition}\label{def:ttn}
Let $\mathcal{A}\in K^{n_1\times\cdots\times n_d}$. A \textbf{tree tensor
network (TTN) decomposition} of $\mathcal{A}$ on $T$ with \textbf{TTN
sizes} $(r_v)_{v\in\mathcal{V}(T)\setminus\{v_0\}}$ is a factorization
\begin{equation}\label{eq:ttn}
  \mathcal{A}(i_1,\dots,i_d)
  =\sum_{\mathbf{j}}\ \prod_{v\in\mathcal{V}(T)}
   \mathcal{C}_v\big(\mathbf{k}_v,\,j_v\big),
\end{equation}
where the internal edge between $v$ and its parent has size $r_v$ and
carries the index $j_v$, the external edge $\mu$ has size $n_\mu$ and
carries the index $i_\mu$, and $\mathbf{k}_v$ lists the indices of the
edges below $v$ from left to right. The sum runs over all
$\mathbf{j}=(j_v)_{v\in\mathcal{V}(T)\setminus\{v_0\}}$ with
$1\le j_v\le r_v$, and the product is taken in the order of $T$. For the
root, the index $j_{v_0}$ is omitted.
\end{definition}

We write $\mathcal{A}=[\![(\mathcal{C}_v)_{v\in\mathcal{V}(T)}]\!]$ 
and call this representation, stored as its cores, a \textbf{TTN tensor}. 

\begin{example}\label{ex:ttn}
For the tree in Figure~\ref{fig:ttn}, the order of $T$ is
$\mathcal{C}_1,\dots,\mathcal{C}_6$, and Eq.~\eqref{eq:ttn} reads
\begin{equation*}
  \mathcal{A}(i_1,\dots,i_5)
  =\sum_{j_1=1}^{r_1}\cdots\sum_{j_5=1}^{r_5}
   \mathcal{C}_1(i_1,j_1)\,\mathcal{C}_2(i_2,i_3,j_2)\,
   \mathcal{C}_3(j_1,j_2,j_3)\,\mathcal{C}_4(i_4,j_4)\,
   \mathcal{C}_5(i_5,j_5)\,\mathcal{C}_6(j_3,j_4,j_5),
\end{equation*}
where $j_k$ is the index of the internal edge between $\mathcal{C}_k$ and its parent.
\end{example}

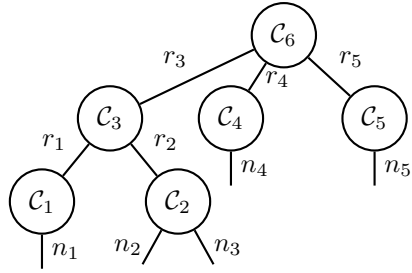
\begin{figure}[h]
\vspace{0.5\baselineskip}
\centering
\setlength{\abovecaptionskip}{4pt}
\begin{tikzpicture}[
  core/.style={circle, draw, thick, minimum size=0.85cm, inner sep=1pt},
  line/.style={thick}
]
  \node[core] (c6) at (3.2,2.2) {$\mathcal{C}_6$};
  \node[core] (c3) at (0.9,1.1) {$\mathcal{C}_3$};
  \node[core] (c4) at (2.5,1.1) {$\mathcal{C}_4$};
  \node[core] (c5) at (4.4,1.1) {$\mathcal{C}_5$};
  \node[core] (c1) at (0,0)     {$\mathcal{C}_1$};
  \node[core] (c2) at (1.8,0)   {$\mathcal{C}_2$};

  \draw[line] (c1) -- node[above left]  {$r_1$} (c3);
  \draw[line] (c2) -- node[above right] {$r_2$} (c3);
  \draw[line] (c3) -- node[above left]  {$r_3$} (c6);
   \draw[line] (c4) -- node[pos=0.4, right] {$r_4$} (c6);
  \draw[line] (c5) -- node[above right] {$r_5$} (c6);

  \draw[line] (c1.south) -- ++(0,-0.45)
    node[midway, right] {$n_1$};
  \draw[line] (c2.-120) -- ++(-0.25,-0.45)
    node[midway, left] {$n_2$};
  \draw[line] (c2.-60) -- ++(0.25,-0.45)
    node[midway, right] {$n_3$};
  \draw[line] (c4.south) -- ++(0,-0.45)
    node[midway, right] {$n_4$};
  \draw[line] (c5.south) -- ++(0,-0.45)
    node[midway, right] {$n_5$};
\end{tikzpicture}
\caption{A tree tensor network with $d=5$. Each edge is labeled by its
size, and the cores are numbered in the order of the tree.}
\label{fig:ttn}
\end{figure}

Algorithm~\ref{alg:ttn-svd} computes a TTN decomposition of a tensor over $K$.
For each vertex $v\ne v_0$, in the order of $T$, it cuts the internal edge between $v$ and its parent. 
In each iteration, $Y$ is unfolded along this cut, with rows indexed by the edges below $v$ and columns indexed by the remaining modes. 
We write $s_1,\dots,s_m$ for the sizes of the edges below $v$, from left to right, and $N_v$ for the number of columns. 
Figure~\ref{fig:ttn-steps} illustrates this idea. 

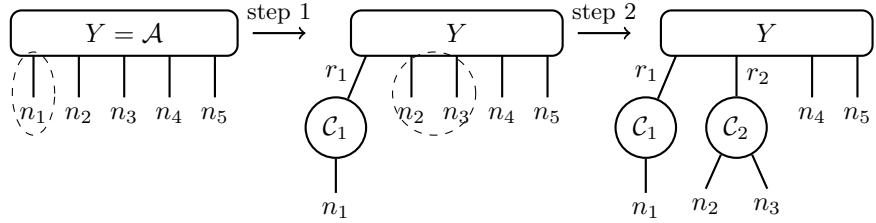
\begin{figure}[h]
\vspace{0.5\baselineskip}
\centering
\setlength{\abovecaptionskip}{4pt}
\begin{tikzpicture}[
  core/.style={circle, draw, thick, minimum size=0.8cm, inner sep=1pt},
  rem/.style={rectangle, rounded corners, draw, thick,
              minimum height=0.6cm},
  line/.style={thick},
  cut/.style={dashed}
]
  \begin{scope}
    \node[rem, minimum width=3.0cm] at (1.2,0) {$Y=\mathcal{A}$};
    \foreach \x/\k in {0/1, 0.6/2, 1.2/3, 1.8/4, 2.4/5}
      \draw[line] (\x,-0.3) -- (\x,-0.85) node[below] {$n_{\k}$};
    \draw[cut] (0,-0.8) ellipse (0.28 and 0.55);
  \end{scope}

  \draw[->, thick] (2.9,0) -- node[above] {\small step 1} (3.6,0);

  \begin{scope}[xshift=4.0cm]
    \node[rem, minimum width=2.8cm] at (1.6,0) {$Y$};
    \node[core] (a1) at (0,-1.25) {$\mathcal{C}_1$};
    \draw[line] (a1) -- node[pos=0.6, left] {$r_1$} (0.4,-0.3);
    \draw[line] (a1.south) -- ++(0,-0.45) node[below] {$n_1$};
    \foreach \x/\k in {1.0/2, 1.6/3, 2.2/4, 2.8/5}
      \draw[line] (\x,-0.3) -- (\x,-0.85) node[below] {$n_{\k}$};
    \draw[cut] (1.3,-0.8) ellipse (0.55 and 0.55);
  \end{scope}

  \draw[->, thick] (7.2,0) -- node[above] {\small step 2} (7.9,0);

  \begin{scope}[xshift=8.1cm]
    \node[rem, minimum width=2.8cm] at (1.6,0) {$Y$};
    \node[core] (b1) at (0,-1.25)   {$\mathcal{C}_1$};
    \node[core] (b2) at (1.2,-1.25) {$\mathcal{C}_2$};
    \draw[line] (b1) -- node[pos=0.6, left] {$r_1$} (0.4,-0.3);
    \draw[line] (b2) -- node[right] {$r_2$} (1.2,-0.3);
    \draw[line] (b1.south) -- ++(0,-0.45) node[below] {$n_1$};
    \draw[line] (b2.-120) -- ++(-0.2,-0.45) node[below] {$n_2$};
    \draw[line] (b2.-60)  -- ++(0.2,-0.45)  node[below] {$n_3$};
    \foreach \x/\k in {2.2/4, 2.8/5}
      \draw[line] (\x,-0.3) -- (\x,-0.85) node[below] {$n_{\k}$};
  \end{scope}
\end{tikzpicture}
\caption{The first two steps of Algorithm~\ref{alg:ttn-svd} on the tree of Figure~\ref{fig:ttn}. Each edge is labeled by its size, and the
dashed curve encloses the edges that index the rows of the next reshape.}
\label{fig:ttn-steps}
\end{figure}

\begin{algorithm}[h!]
\caption{TTN-SVD$_K$}\label{alg:ttn-svd}
\KwIn{$\mathcal{A}\in K^{n_1\times\cdots\times n_d}$; rooted tree $T$;
      target TTN sizes $(r_v)_{v\in\mathcal{V}(T)\setminus\{v_0\}}$ or
      relative error tolerance $\varepsilon\ge0$}
\KwOut{cores $(\mathcal{C}_v)_{v\in\mathcal{V}(T)}$ of a TTN approximation
       of $\mathcal{A}$ with TTN sizes $(r_v)$}
$Y\leftarrow\mathcal{A}$; $\ b\leftarrow\#\mathcal{V}(T)-1$\;
\For{$v\in\mathcal{V}(T)\setminus\{v_0\}$ in the order of~$T$ }{
  $\bar{Y}\leftarrow\operatorname{reshape}\big(Y,\
     (s_1\cdots s_m)\times N_v\big)$
    \tcp*{cut the edge from $v$ to its parent}
  $C_v\leftarrow\mathrm{LLSV}_K\big(\bar{Y},\ r_v\text{ or }
     \varepsilon\|\mathcal{A}\|/\sqrt{b}\big)$\;
     $Y\leftarrow C_v^H\bar{Y}$\;
  $\mathcal{C}_v\leftarrow\operatorname{reshape}\big(C_v,\
     s_1\times\cdots\times s_m\times r_v\big)$\;
  
}
$\mathcal{C}_{v_0}\leftarrow\operatorname{reshape}\big(Y,\
   s_1\times\cdots\times s_m\big)$\;
\Return{$(\mathcal{C}_v)_{v\in\mathcal{V}(T)}$}
\end{algorithm}

\paragraph{Computational complexity of TTN-SVD$_K$.} 
As for TT-SVD$_K$, we apply $\Pi$ to $\mathcal{A}$ once and work on the components. 
Let $M_v=s_1\cdots s_m$ be the number of rows of $\bar{Y}$ at vertex $v$. 
The LLSV of $\bar{Y}\in K^{M_v\times N_v}$ costs $O(M_vN_v\min(M_v,N_v)\,\kappa_K)$, 
and the product $C_v^H\bar{Y}$, $O(r_vM_vN_v\,\kappa_K)$, is no larger since $r_v\le\min(M_v,N_v)$. 
The reshapes only rearrange the $M_vN_v$ entries of $Y$. The total cost is
\[
  O\Big(N n_K^2+\kappa_K\sum_{v\in\mathcal{V}(T)\setminus\{v_0\}}
  M_vN_v\min(M_v,N_v)\Big).
\]
On the TT tree, $M_{v}=r_{k-1}n_k$ and $N_{v}=P_k$ for the vertex $v$ of $\mathcal{C}_k$, and this is the cost of TT-SVD$_K$.\\

For commutative $K$, Algorithm~\ref{alg:ttn-svd} is quasi-optimal for every tree, with a factor given by the number $\#\mathcal{V}(T)-1$ of
internal edges.

\begin{theorem}[Quasi-optimality of TTN-SVD$_K$]\label{thm:ttn}
Let $K$ be a finite-dimensional, unital, very proper, commutative $*$-algebra with an inner product that satisfies the adjoint identities.
Let $T$ be a rooted tree as above, and let $\mathcal{T}$ be the TTN tensor computed from $\mathcal{A}\in K^{n_1\times\cdots\times n_d}$ by
Algorithm~\ref{alg:ttn-svd}. 
If the algorithm is run with TTN sizes $(r_v)_{v\in\mathcal{V}(T)\setminus\{v_0\}}$, then
\begin{equation}\label{eq:ttn-bound}
  \|\mathcal{A}-\mathcal{T}\|
  \le\sqrt{\#\mathcal{V}(T)-1}\,\|\mathcal{A}-\mathcal{T}^\star\|,
\end{equation}
where $\mathcal{T}^\star$ is a best TTN approximation of $\mathcal{A}$ on $T$ with those TTN sizes.
If the root has no external edges and exactly two children $u$ and $w$ with $r_u=r_w$, then $\sqrt{\#\mathcal{V}(T)-1}$ can be replaced by $\sqrt{\#\mathcal{V}(T)-2}$.
\end{theorem}

We prove Theorem~\ref{thm:ttn} in Appendix~\ref{app: proof of corollaries}. Theorem~\ref{thm:ttn} gives the bound $\sqrt{d-1}$ of
Theorem~\ref{cor: optimality for TT-SVD} on the TT tree, which has $d$ vertices, and the bound $\sqrt{d}$ of Theorem~\ref{cor: optimality for Tucker decompositions} on the Tucker tree, which has $d+1$ vertices.
The trees of the hierarchical Tucker decomposition have $2d-1$ vertices and a root with two children of equal
size~\cite[Eq.~(11.51)]{hackbusch2019}, and Theorem~\ref{thm:ttn} gives the bound $\sqrt{2d-3}$ of~\cite[Thm.~11.66]{hackbusch2019}.

\section{Numerical Examples}\label{sec:experiments}
In this section, we compare decompositions over $K$ with standard decompositions over $\R$ or $\C$ on simulated and real data. The \textbf{compression ratio} is the number of floating-points numbers needed to store the data divided by the number of floating-points numbers needed to store its
factored approximation, counting a complex number as two real numbers. For the decompositions over $K$, we store the factors in Wedderburn–Artin coordinates with component-specific ranks, omitting zero
padding; reported storage counts refer to this representation. The larger the ratio, the stronger the compression. The \textbf{relative error} of an approximation $\widehat\calA$ of $\calA$ is $\|\calA-\widehat\calA\|/\|\calA\|$. In each experiment, we plot the compression ratio against the relative error. 

Our code is available at \url{https://github.com/paulinahoyos/tensor_decompositions}.

\subsection{Quantum-circuit data}
We compress simulated data from noisy three-qubit quantum circuits, which are represented by elements of $K=\M_8(\C)$~\cite{NielsenChuang2010}. 

We split each circuit into two halves. Each half consists of four two-qubit gates, acting alternately on qubits $1$ and $2$ and on qubits $2$ and $3$, each drawn from the Haar measure on the $4\times4$ unitary group; this is a standard model of random quantum circuits~\cite{BrandaoHarrowHorodecki2016,FisherKhemaniNahumVijay2023}. Real gates carry small coherent errors driven by a slowly varying calibration offset shared by all gates~\cite{WallmanEmerson2016}; we model this by following each gate with a small rotation error whose amplitude is proportional to a common offset $\xi\sim\mathcal{N}(0,1)$.

We generate $500$ random circuits and arrange the data as a $500\times500$ matrix $A$ over $K$ with entries
\[
  A_{ij} = \mathbb{E}_{\xi}\bigl[U^{(2)}_i(\xi)\,U^{(1)}_j(\xi)\bigr],
\]
where $U^{(2)}_i(\xi)$ and $U^{(1)}_j(\xi)$ are the unitaries of the $i$-th second half and the $j$-th first half, and the expectation is computed by Gauss--Hermite quadrature with $20$ nodes. For a fixed $\xi$, $A$ has rank one over $K$; since the errors are small, the terms for different $\xi$ are close, and $A$ is close to a matrix of low rank over $K$.

We compare the SVD over $K$ with the SVDs of the mode-$1$ and mode-$2$ unfoldings of $A$, viewed as a $500\times500\times64$ complex tensor, and with the HOSVD of that tensor (Figure~\ref{fig:circuits}). At every relative error between $10^{-7}$ and $10^{-2}$, the SVD over $K$ compresses roughly $5$ to $15$ times more than the HOSVD and $30$ times more than the SVDs of the unfoldings, and it reaches the single-precision accuracy of the data with about four columns over $K$.
\begin{figure}[h!]
    \centering
    \includegraphics[scale=0.9]{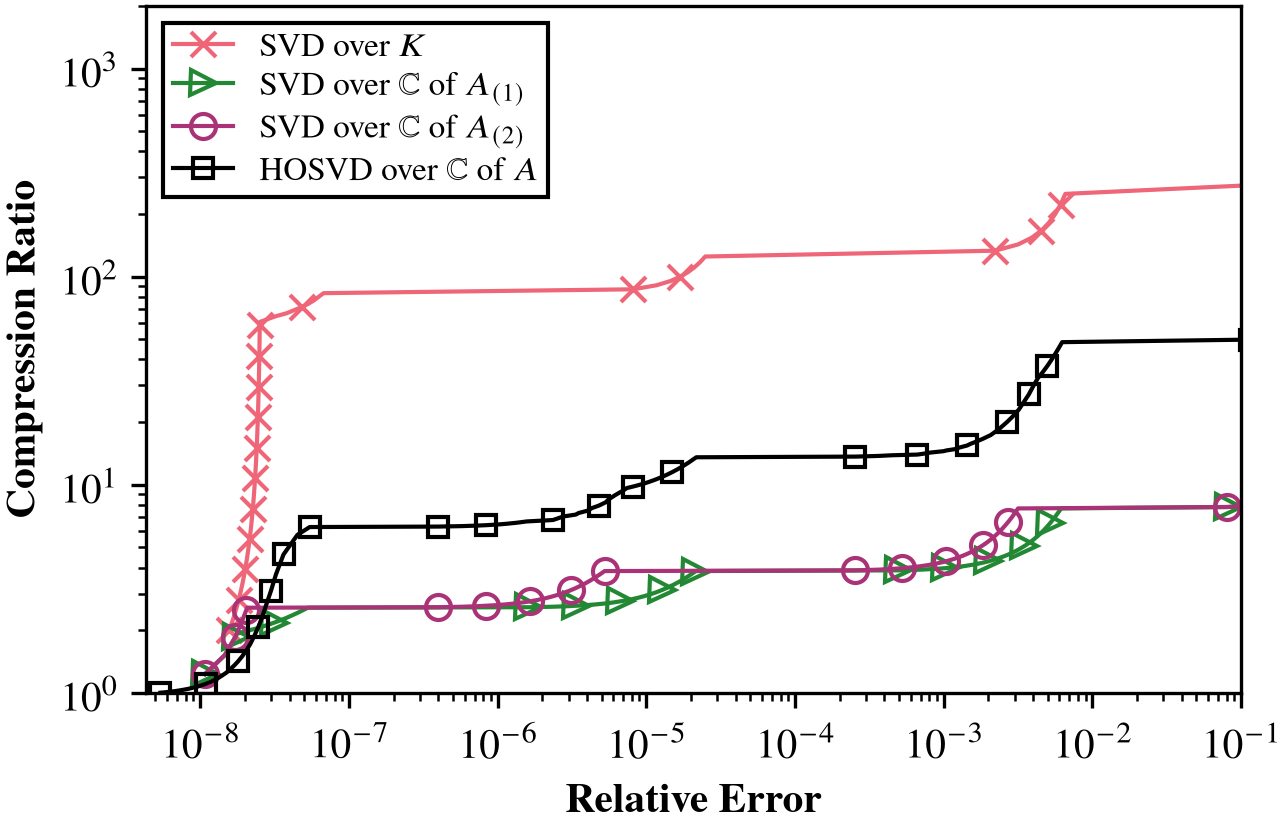}
    \caption{Compression ratio against relative error for the quantum-circuit data. The SVD over $K=\M_8(\C)$ compresses roughly an order of magnitude more than the HOSVD over $\C$, and more than the SVDs of the unfoldings $A_{(1)}$ and $A_{(2)}$, at every error above the single-precision floor.}
\label{fig:circuits}
\end{figure}

\subsection{Vlasov--Poisson data}
We compress simulated solutions of the  Vlasov--Poisson equations from kinetic theory, which describe a collisionless plasma in which electrons interact only through the electric field they create. The phase-space density of the electrons $f(x_1,x_2,v_1, v_2, t)$ is transported by the electric field $E(x_1,x_2,t)$, which is in turn determined by the electron density through Gauss's law: 
\begin{equation}\label{eq: vlasov-poisson}
    \partial_t f + v \cdot \nabla_x f - E \cdot \nabla_v f = 0, \qquad \nabla_x \cdot E = 1 - \int f \, dv.
\end{equation}
On a periodic box, these equations commute with translations in $x_1$ and $x_2$, so we take $K=\R[\Z_{32}\times\Z_{32}]$, the group algebra of translations of the $32\times32$ position grid.

Using the splitting method of~\cite{ChengKnorr1976}, we solve the equations numerically with the initial condition of the nonlinear Landau damping test (\cite{FilbetSonnendrucker2003,Kormann2015,EinkemmerLubich2018})
\[
  f(x_1, x_2, v_1, v_2, 0) = \frac{1}{2\pi}e^{-|v|^2/2}\bigl(1+\alpha\cos(kx_1)+\alpha\cos(kx_2)\bigr),
  \qquad \alpha=0.5,\ k=0.5,
\]
on the periodic box $x\in[0,4\pi)^2$, with velocities truncated to $[-6,6)^2$. We discretize position on a $32\times32$ grid and velocity on a $48\times48$ grid, and take $32$ snapshots at $t=1.25,2.5,\dots,40$.
We thus obtain a tensor $\calF$ of size ${48 \times 48 \times 32}$ over $K$, with entries 
\[ \calF(i, j, k) = f(\cdot, \cdot, v_x(i), v_y(j), t(k)), \]
where $v_x(i)$, $v_y(j)$ and $t(k)$ denote the $i$-th and $j$-th velocity grid points and the $k$-th snapshot time.

Since $K$ is commutative, we compress $\calF$ with both Tucker decompositions, computed by $\mathrm{HOSVD_K}$ and $\mathrm{ST\text{-}HOSVD}_K$. 
We compare these decompositions over $K$ with the same decompositions of the real $32\times32\times48\times48\times32$ tensor,  and with the SVD of its unfolding $F_{(x,y)}$, whose rows are indexed by position and whose columns are indexed by velocity and time (Figure~\ref{fig:vlasov}). The Tucker decompositions over $K$ compress best, with a compression ratio of about $250$ at relative error $0.1$, against $70$ for TT over $K$ and at most $50$ over $\R$. 

\begin{figure}[h!]
    \centering
    \includegraphics[scale=0.8]{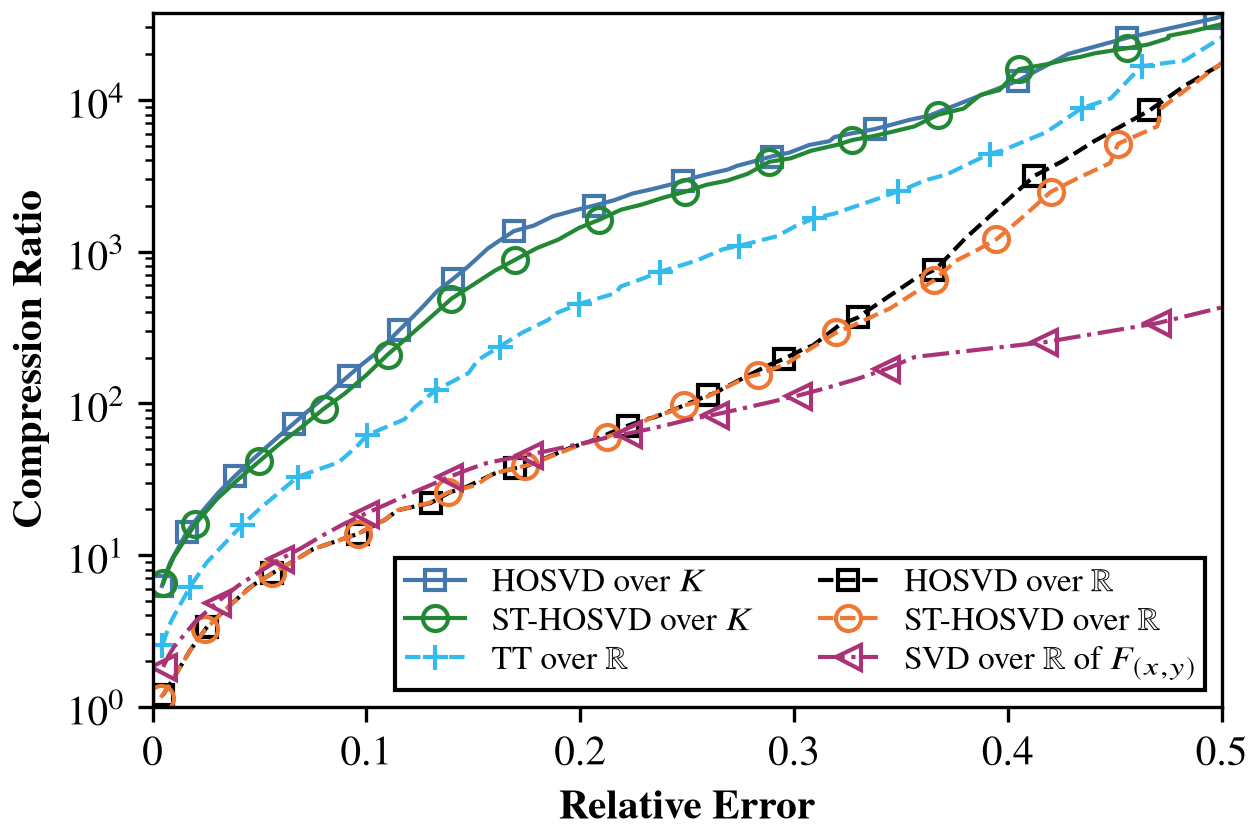}
    \caption{Compression ratio against relative error for the Vlasov--Poisson data. The Tucker decompositions over $K=\R[\Z_{32}\times\Z_{32}]$ compress about an order of magnitude more than those over $\R$ at small and moderate relative errors.}
\label{fig:vlasov}
\end{figure}

\subsection{ERA5 climate data}
We use ERA5~\cite{Hersbach2020}, a reanalysis of the Earth’s atmosphere\footnote{ 
A reanalysis combines a weather model with past observations to give the best estimate of the global state of the atmosphere at each time}
and the standard training data for machine-learning weather models~\cite{Rasp2020WB,Rasp2024WB2}.
We download ERA5 from the WeatherBench~2 archive~\cite{Rasp2024WB2}, on  a $64\times32$ grid in longitude and latitude with spacing $5.625^\circ$, and select five variables (geopotential, temperature, the two horizontal wind components and specific humidity) at
$13$ pressure levels, once a day at 00\,UTC for the $365$ days of 2019.
From each value we subtract the 1990--2019 climatology for that day of the year, which removes the seasonal
cycle, and we divide each pair of variable and level by the standard deviation of its anomalies. 
We thus obtain a real tensor $\calA$ of size $64\times32\times13\times5\times365$ with entries
\[
  \calA(i_1,i_2,j,k,l) = a\bigl(\lambda(i_1),\phi(i_2),p(j),k,d(l)\bigr),
\]
where $a$ denotes the standardized anomalies, $\lambda(i_1)$ and $\phi(i_2)$ are the $i_1$-th longitude and $i_2$-th latitude, $p(j)$ is the $j$-th pressure level, $k$ indexes the five variables, and $d(l)$ is the $l$-th day.

\paragraph{Order-$4$ tensor over $K_1=\R[\Z_{64}]$.}
Longitude is periodic, and rotating the grid about the Earth's axis by one step, $5.625^\circ$, generates the cyclic group $\Z_{64}$. Putting longitude into the algebra gives an order-$4$ tensor $\calB_1$ over $K_1=\R[\Z_{64}]$ with entries
\[ \calB_1(i_2,j,k,l) = \calA(\cdot,i_2,j,k,l). \]
We compress $\calB_1$ with ST-HOSVD, TT-SVD  and TTN-SVD over $K_1$, and compare them with the same decompositions of $\calA$ over $\R$ (Figure~\ref{fig:era5-lon}). The TTN decompositions use the trees in
Figure~\ref{fig:era5-trees}. At relative error $0.3$, every decomposition over $K_1$ compresses $1.7$ to $2.8$ times more than its counterpart over $\R$, and the gap widens at larger errors for TT and TTN. 
 At relative error $0.3$, TTN over $K_1$ reaches a compression ratio of about $19$ and TT over $K_1$ about $15$, against about $8$ for the best decomposition over $\R$, which is TTN.\\[2mm]

\begin{figure}[h]
\centering
\begin{tikzpicture}[
  core/.style={circle, draw, thick, minimum size=0.8cm, inner sep=1pt},
  line/.style={thick},
  lab/.style={font=\small}
]
  \node[core] (a4) at (2.0,2.3)  {$\mathcal{C}_4$};
  \node[core] (a3) at (1.1,1.15) {$\mathcal{C}_3$};
  \node[core] (a1) at (0.1,0)    {$\mathcal{C}_1$};
  \node[core] (a2) at (2.1,0)    {$\mathcal{C}_2$};

  \draw[line] (a1) -- (a3);
  \draw[line] (a2) -- (a3);
  \draw[line] (a4) -- node[lab, above left] {space} (a3);

  \draw[line] (a1.south) -- ++(0,-0.5)
    node[lab, midway, left] {lat};
  \draw[line] (a2.-120) -- ++(-0.3,-0.5)
    node[lab, midway, left] {pressure};
  \draw[line] (a2.-60) -- ++(0.3,-0.5)
    node[lab, midway, right] {variables};
  \draw[line] (a4.-50) -- ++(0.45,-0.55)
    node[lab, midway, right] {time};

  \begin{scope}[xshift=6.0cm]
    \node[core] (b5) at (3.1,2.3)  {$\mathcal{C}_5$};
    \node[core] (b4) at (2.2,1.15) {$\mathcal{C}_4$};
    \node[core] (b1) at (0,0)      {$\mathcal{C}_1$};
    \node[core] (b2) at (1.5,0)    {$\mathcal{C}_2$};
    \node[core] (b3) at (3.4,0)    {$\mathcal{C}_3$};

    \draw[line] (b1) -- (b4);
    \draw[line] (b2) -- (b4);
    \draw[line] (b3) -- (b4);
    \draw[line] (b5) -- node[lab, above left] {space} (b4);

    \draw[line] (b1.south) -- ++(0,-0.5)
      node[lab, midway, left] {lon};
    \draw[line] (b2.south) -- ++(0,-0.5)
      node[lab, midway, left] {lat};
    \draw[line] (b3.-120) -- ++(-0.3,-0.5)
      node[lab, midway, left] {pressure};
    \draw[line] (b3.-60) -- ++(0.3,-0.5)
      node[lab, midway, right] {variables};
    \draw[line] (b5.-50) -- ++(0.45,-0.55)
      node[lab, midway, right] {time};
  \end{scope}
\end{tikzpicture}
\caption{Trees of the TTN decompositions in Figure~\ref{fig:era5-lon} over $K_1$ (left) and over
$\mathbb{R}$ (right). External edges are labeled by mode, and the cores are
numbered in the order of the tree.}
\label{fig:era5-trees}
\end{figure}
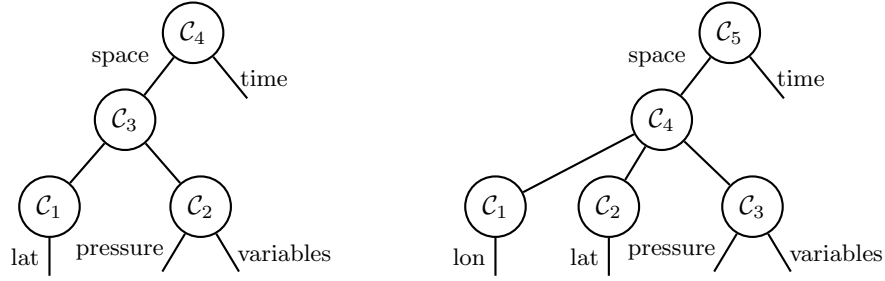

\begin{figure}[h!]
    \centering
    \includegraphics[scale=0.9]{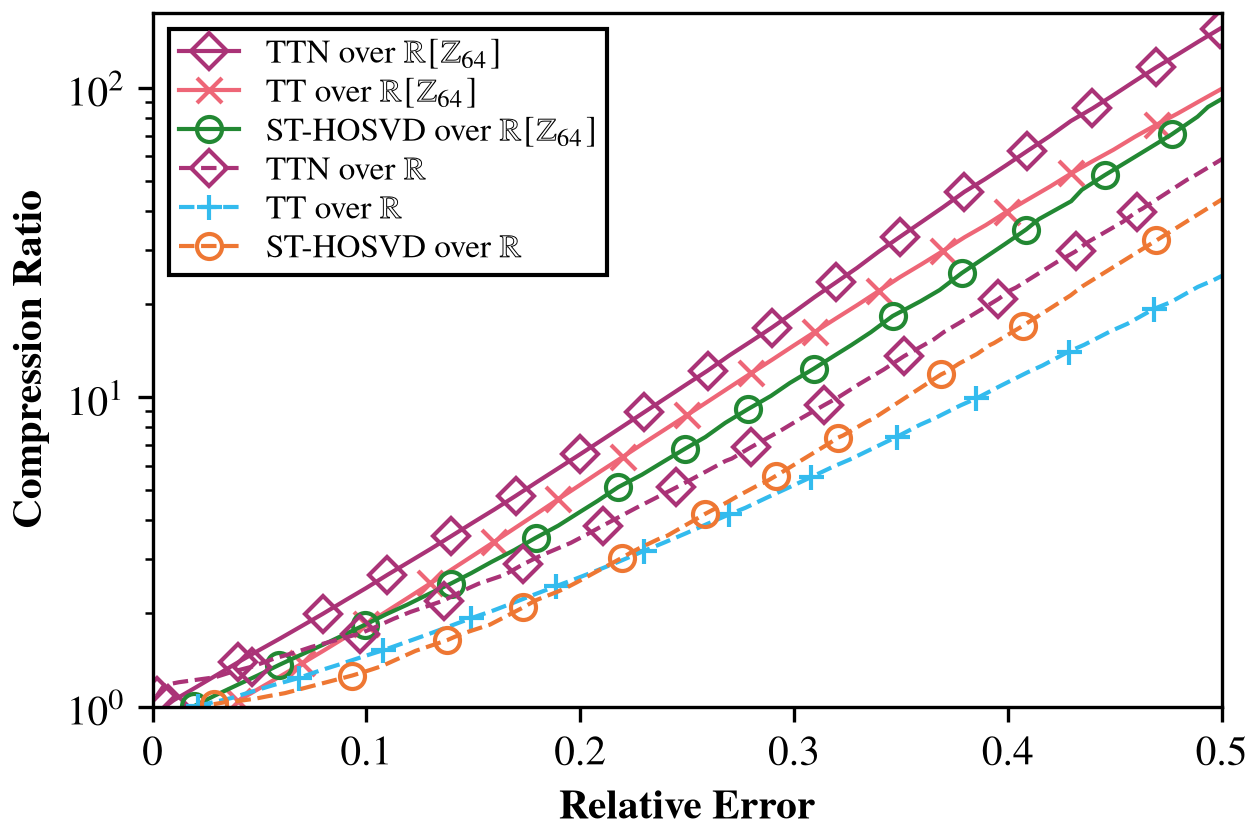}
    \caption{Compression ratio against relative error for the ERA5 data with longitude in the algebra, $K_1=\R[\Z_{64}]$.}
    \label{fig:era5-lon}
\end{figure}

\paragraph{Order-$3$ tensor over $K_2=\R[\Z_{64}]\otimes\mathrm{DCT}_{32}$.}
We next put latitude into the algebra as well. Latitude runs from pole to pole and is not periodic, so we use $\mathrm{DCT}_{32}$, the $\star_M$ algebra in Example~\ref{ex: tubal algebras} with $M$ the orthogonal discrete cosine transform matrix, the standard transform for smooth non-periodic signals. Over $K_2=\R[\Z_{64}]\otimes\mathrm{DCT}_{32}$, each entry is a longitude--latitude map, and the data form an order-$3$ tensor $\calB_2$ with entries
\[ \calB_2(j,k,l) = \calA(\cdot,\cdot,j,k,l). \]
 We compress $\calB_2$ with TT-SVD, HOSVD and ST-HOSVD over $K_2$ and compare them with the same decompositions over $\R$ (Figure~\ref{fig:era5-lonlat}). TT and ST-HOSVD over $K_2$ perform almost identically: both compress more than HOSVD over $K_2$, and two to three times more than the decompositions over $\R$.
\begin{figure}[h!]
    \centering
    \includegraphics[scale=0.9]{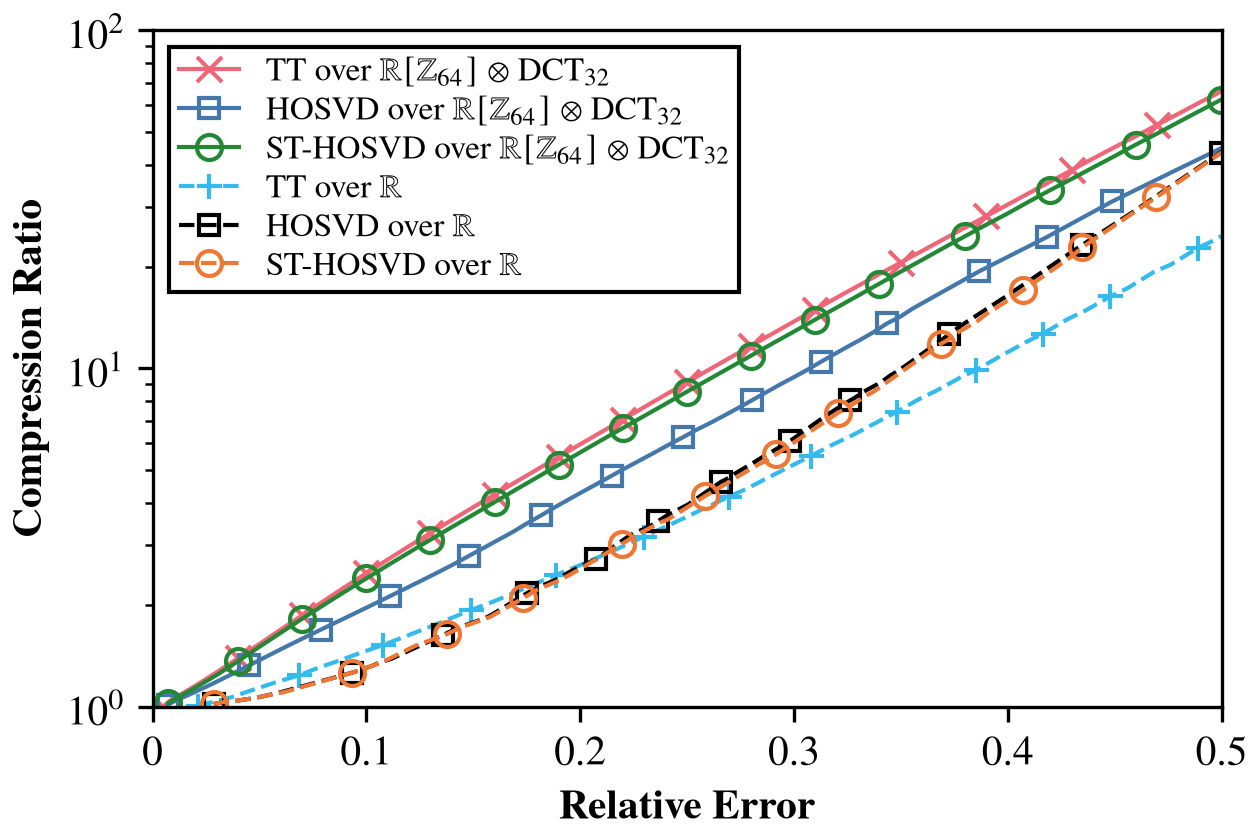}
    \caption{Compression ratio against relative error for the ERA5 data with longitude and latitude in the algebra, $K_2=\R[\Z_{64}]\otimes\mathrm{DCT}_{32}$.}
    \label{fig:era5-lonlat}
\end{figure}

\section{Conclusions}
In this paper, we developed a unified theory of the singular value decomposition over finite-dimensional $\ast$-algebras. 
We characterized when every matrix admits an SVD and when ordered SVD truncations give optimal low-rank approximations as in the Eckart--Young theorem. 
Our theory brings together constructions previously developed one algebra at a time, and yields a family of new decomposition formats for real and complex tensors.

Computationally, we extended SVD-based algorithms for Tucker, tensor train, and tree tensor network decompositions to $\ast$-algebras. 
The tensor train SVD algorithm is quasi-optimal in terms of approximation error for noncommutative algebras satisfying our conditions.  
Meanwhile, approximation guarantees for Tucker and tree tensor networks hold in the commutative case.  
Computational experiments on datasets from quantum circuit, plasma, and climate applications showed that the new decompositions can achieve better compression than classical approaches.

Several questions remain open. One is whether quasi-optimal Tucker and tree tensor network algorithms can be developed for noncommutative algebras. 
Another important question is how to choose or learn an appropriate algebra when the dataset doesn't suggest a clear choice of algebra.  
For the $\star_M$-product, Newman and Keegan~\cite{newman2024optimal} jointly learn an orthogonal transform $M$ and a tensor representation using variable projection.  
Their work could offer ideas on how to learn more general $\ast$-algebra structures from data.  
Finally, we think that extending our framework to randomized SVD-based algorithms is another promising  direction.

\subsection*{Acknowledgments}
This research was supported in part by NSF DMS 2309782, NSF DMS 2436499, NSF CISE-IIS 2312746, DE SC0025312, and the Sloan Foundation.

\subsection*{AI Use Disclosure}
The main ideas of this work are the authors' own: extending the SVD to matrices over arbitrary $\ast$-algebras, the sufficient condition for
its existence given by the Wedderburn--Artin decomposition,  the necessary and sufficient condition for the optimality of truncated SVDs, the extension to tensor decompositions, and the design of all algorithms in the paper.
We used Claude for literature search, which helped us find the notions of nonnegative cone and very proper $\ast$-algebras. 
In discussions driven by the authors' questions, Claude generated examples, helped formulate the notion of ordered SVDs and their characterization, and suggested proofs of Propositions~\ref{prop: characterization of ordered SVDs} and~\ref{prop:rank}, which we then rewrote.
For the manuscript, we used Claude and ChatGPT to double-check proofs, revise the text for notation, consistency and clarity, and produce the TikZ figures. 
Finally,  Claude drafted and debugged code implementing our  algorithms and wrote its documentation.
Each line of AI-assisted work has been examined and revised by the authors, and we take  responsibility for the final content of this
work.

\bibliography{references.bib}
\bibliographystyle{abbrvurl}

\appendix

\section{Proof of Proposition~\ref{prop: characterization of ordered SVDs}}\label{app: proof for ordered SVDs}
Recall from Section~\ref{sec: algorithm for svd} that $a\in K^+$ if and only if every $\pi_\rho(a)$ is Hermitian positive semidefinite. Moreover, $c\in K$ is self-adjoint and central if and only if $\pi_\rho(c)=\alpha_\rho I_{d_\rho}$ with $\alpha_\rho\in\R$ for every $\rho$, and the $\alpha_\rho$ can be chosen independently.

\smallskip\noindent\underline{$(1)\Leftrightarrow(2)$}: Let $A=U\Sigma V^H$ be an SVD, and fix $1\le i<s$. By the facts above, there is a self-adjoint central element $c_i$ with $\Sigma_{ii}\succeq c_i\succeq\Sigma_{i+1,i+1}$ if and only if, for every $\rho$, there is $\alpha_{i\rho}\in\R$ such that
\begin{equation}
    \pi_\rho(\Sigma_{ii})-\alpha_{i\rho}I_{d_\rho}
    \quad\text{and}\quad
    \alpha_{i\rho}I_{d_\rho}-\pi_\rho(\Sigma_{i+1,i+1})
\end{equation}
are Hermitian positive semidefinite, that is, all their eigenvalues are nonnegative. Since $\pi_\rho(\Sigma_{ii})$ and $\pi_\rho(\Sigma_{i+1,i+1})$ are positive semidefinite, their eigenvalues are their singular values, so this holds exactly when $\sigma_{\min}(\pi_\rho(\Sigma_{ii}))\ge\alpha_{i\rho}\ge\sigma_{\max}(\pi_\rho(\Sigma_{i+1,i+1}))$. Such an $\alpha_{i\rho}$ exists if and only if Eq.~\eqref{eq: ordering of singular values under pi} holds.

\smallskip\noindent\underline{$(2)\Rightarrow(3)$}: Let $A=U\Sigma V^H$ be an SVD satisfying Eq.~\eqref{eq: ordering of singular values under pi}. Each $\pi_\rho(\Sigma_{ii})$ is Hermitian positive semidefinite, so by the spectral theorem there is a unitary matrix $W_{i\rho}$ such that $W_{i\rho}^H\pi_\rho(\Sigma_{ii})W_{i\rho}$ is diagonal, with the singular values of $\pi_\rho(\Sigma_{ii})$ in nonincreasing order. Let $w_i=\pi^{-1}(\bigoplus_\rho W_{i\rho})$, which is unitary, and let $W_m$ and $W_n$ be as in the proposition. Then $UW_m$ and $VW_n$ are unitary, and $W_m^H\Sigma W_n$ is diagonal with entries $w_i^\ast\Sigma_{ii}w_i\in K^+$, so $A=(UW_m)(W_m^H\Sigma W_n)(VW_n)^H$ is an SVD. For each $\rho$, $\Pi_\rho(W_m^H\Sigma W_n)$ is diagonal, and its diagonal is the concatenation of the sorted singular values of $\pi_\rho(\Sigma_{11}),\dots,\pi_\rho(\Sigma_{ss})$, followed by zeros; by Eq.~\eqref{eq: ordering of singular values under pi}, it is nonincreasing. Hence each $\widehat A_\rho=\Pi_\rho(UW_m)\,\Pi_\rho(W_m^H\Sigma W_n)\,\Pi_\rho(VW_n)^H$ is a classical SVD with nonincreasing singular values, and the SVD is induced by $\pi$.

\smallskip\noindent\underline{$(3)\Rightarrow(2)$}: Assume that $A=(UW_m)(W_m^H\Sigma W_n)(VW_n)^H$ is an SVD induced by $\pi$. For each $\rho$, $\Pi_\rho(W_m^H\Sigma W_n)$ is diagonal with nonincreasing nonnegative entries. Its $i$-th diagonal block is $\pi_\rho(w_i^\ast\Sigma_{ii}w_i)=\pi_\rho(w_i)^H\pi_\rho(\Sigma_{ii})\pi_\rho(w_i)$, which is unitarily similar to $\pi_\rho(\Sigma_{ii})$, so its entries are the singular values of $\pi_\rho(\Sigma_{ii})$. Since the whole diagonal is nonincreasing, $\sigma_{\min}(\pi_\rho(\Sigma_{ii}))\ge\sigma_{\max}(\pi_\rho(\Sigma_{i+1,i+1}))$ for each $1\le i<s$.

\section{Proof of Proposition~\ref{prop:rank}}\label{app: proof for rank}
Write $R=\max_{\rho}\lceil r_\rho/d_\rho\rceil$, and, for an SVD $A=U\Sigma V^H$, let $n(\Sigma)$ be the number of nonzero diagonal entries of $\Sigma$. We prove that
\[
  \mathrm{rank}_K(A)\ \ge\ R\ =\ n(\widetilde\Sigma)\ \ge\ \min n(\Sigma)\ \ge\ \mathrm{rank}_K(A),
\]
where $\widetilde\Sigma$ comes from an arbitrary ordered SVD and the minimum runs over all SVDs of $A$.

Let $k=\mathrm{rank}_K(A)$, so that $A=PQ$ with $P\in K^{m\times k}$ and $Q\in K^{k\times n}$. Then $\widehat A_\rho=\Pi_\rho(P)\Pi_\rho(Q)$, where $\Pi_\rho(P)$ has $kd_\rho$ columns, so $r_\rho\le kd_\rho$. Hence $k\ge\lceil r_\rho/d_\rho\rceil$ for every $\rho$, and $k\ge R$.

Let $s=\min(m,n)$, and denote the singular values of $\widehat A_\rho$ by $\sigma_1^{(\rho)}\ge\cdots\ge\sigma_{sd_\rho}^{(\rho)}\ge0$. Let $A=\widetilde U\widetilde\Sigma\widetilde V^H$ be an ordered SVD. By Proposition~\ref{prop: characterization of ordered SVDs}, conjugating each $\widetilde\Sigma_{ii}$ by a unitary element gives an SVD induced by $\pi$, in which the diagonal of the $\rho$-th component is $(\sigma_1^{(\rho)},\dots,\sigma_{sd_\rho}^{(\rho)})$. Conjugation preserves singular values, so the singular values of $\pi_\rho(\widetilde\Sigma_{ii})$ are $\sigma_{(i-1)d_\rho+1}^{(\rho)}\ge\cdots\ge\sigma_{id_\rho}^{(\rho)}$. Therefore $\pi_\rho(\widetilde\Sigma_{ii})\ne0$ exactly when $\sigma_{(i-1)d_\rho+1}^{(\rho)}\ne0$, that is, when $(i-1)d_\rho+1\le r_\rho$, or equivalently $i\le\lceil r_\rho/d_\rho\rceil$. Since $\pi$ is injective, $\widetilde\Sigma_{ii}\ne0$ if and only if $\pi_\rho(\widetilde\Sigma_{ii})\ne0$ for some $\rho$, that is, if and only if $i\le R$. Finally, $r_\rho\le sd_\rho$ gives $R\le s$, so $\widetilde\Sigma$ has exactly $R$ nonzero diagonal entries, $\widetilde\Sigma_{11},\dots,\widetilde\Sigma_{RR}$.

Trivially, $n(\widetilde\Sigma)\ge\min n(\Sigma)$. Finally, any SVD $A=U\Sigma V^H$ writes $A$ as the sum of the $n(\Sigma)$ outer products $(u_i\Sigma_{ii})v_i^H$ with $\Sigma_{ii}\ne0$, where $u_i$ and $v_i$ are the $i$-th columns of $U$ and $V$, so $n(\Sigma)\ge\mathrm{rank}_K(A)$.

\section{Proof of Theorems~\ref{cor: optimality for Tucker decompositions},~\ref{cor: optimality for TT-SVD} and~\ref{thm:ttn}}
\label{app: proof of corollaries}

The proofs of Theorems~\ref{cor: optimality for Tucker decompositions}, \ref{cor: optimality for TT-SVD} and~\ref{thm:ttn} follow the proofs for
real tensors: 
the error bound of De Lathauwer et al.~\cite{DeLathauwer2000} for the HOSVD, the error identity and quasi-optimality of Vannieuwenhoven
et al.~\cite{Vannieuwenhoven2012} and Hackbusch~\cite{hackbusch2019} for
the HOSVD and ST-HOSVD, the error identity and quasi-optimality of
Oseledets~\cite{Oseledets2011} for the TT-SVD, and the leaves-to-root truncation of Hackbusch~\cite[Thm.~11.66]{hackbusch2019} for the TTN-SVD. See
also~\cite[Chapters~7 and~8]{BallardKolda2025}. 
These proofs use only the optimality of the truncated SVD, which
holds over $K$ by Theorem~\ref{thm: optimality of rank-k SVD approximation},
and the following properties of matrices over $K$.
 
\begin{proposition}\label{prop: matrix properties}
    Assume that the inner product on $K$ satisfies the adjoint identities (Definition~\ref{axiom for K}). Then matrices over $K$ satisfy the following properties.
    \begin{enumerate}
        \item Adjoint identities:
        \begin{equation}
        \langle AB, C \rangle = \langle B, A^H C \rangle \quad \text{ and } \quad \langle AB, C \rangle = \langle A,  C B^H \rangle
    \end{equation}
    for matrices of compatible sizes $A \in \A^{m \times p}, B \in \A^{p \times n}, C \in \A^{m \times n}$.
         \item Unitary invariance: Given a unitary matrix $U \in \A^{m\times m}$, then
         \begin{align}
             &\|A U\| = \|A\|, 
         \end{align}
         for any matrix $A \in \A^{p\times m}$, and
         \begin{align}
             &\|UB \| = \|B\|
         \end{align}
         for any matrix $B \in K^{m \times n}$.
                 \item Isometry invariance: Given a matrix $U \in \A^{m\times n}$ with $U^HU = I$,
        \begin{equation}
            \|UB\| = \|B\|
        \end{equation}
        for any matrix $B \in K^{n \times p}$.
        \item Orthogonal projector properties: Given a matrix $U \in \A^{m\times n}$ with $UU^H U = U$ (for example, with orthonormal columns), then
        \begin{enumerate}
            \item  For any matrices $A, B \in \A^{m\times p}$,
            \begin{equation}\label{eq: orthogonality of projection residual}
            \langle U U^H A,\;
          (I - U U^H) B \rangle = 0,
        \end{equation}
        \item For $A \in K^{m \times p}$,
        \begin{equation}\label{eq: pythagoras for projector}
            \|A\|^2 = \| U U^H A \|^2 + \|(I - UU^H)A\|^2
        \end{equation}

        \end{enumerate}
       
    \end{enumerate}
\end{proposition}
\begin{proof}
All properties follow directly from the adjoint identities (Definition~\ref{axiom for K}) and the definition of the inner product for matrices over $K$ given in Eq.~\eqref{eq: inner product of matrix}.
\begin{enumerate}
    \item The inner product of the $m \times n$ matrices $AB$ and $C$ is given by
    \begin{align}
\begin{split}
     \langle AB, C \rangle
     &= \sum_{i=1}^m \sum_{j=1}^n \sum_{k=1}^p \langle A_{ik}B_{kj}, C_{ij} \rangle \\
     &= \sum_{i=1}^m \sum_{j=1}^n \sum_{k=1}^p \langle B_{kj}, A_{ik}^\ast C_{ij} \rangle\\
     &= \sum_{i=1}^m \sum_{j=1}^n \sum_{k=1}^p \langle B_{kj}, (A^H)_{ki} C_{ij} \rangle\\
     &= \langle B, A^HC \rangle.
     \end{split}
\end{align}
The same argument applies to the other identity.

\item By the adjoint identities, we have that 
\begin{equation}
    \| AU \|^2 = \langle AU, AU \rangle = \langle A , AU U^H \rangle = \langle A, A \rangle = \|A\|^2,
\end{equation}
and similarly for $B$.

\item By the adjoint identities and $U^HU = I$,
\begin{equation}
    \|UB\|^2 = \langle UB, UB \rangle = \langle B, U^HUB \rangle = \langle B, B \rangle = \|B\|^2 .
\end{equation}

\item Since $UU^H U =U$, it follows that
\begin{align}
\begin{split}
    \langle U U^H A,\;(I - U U^H) B \rangle 
    &= \langle A, UU^H (I - UU^H) B \rangle \\
    &= \langle A, (UU^H - U(U^HU)U^H) B \rangle \\
    &= \langle A, (UU^H - UU^H) B \rangle \\
    &=0.
    \end{split}
\end{align}
Taking $B = A$, and using also
$\langle (I - UU^H)A,\, UU^H A \rangle = \overline{\langle UU^H A,\, (I - UU^H)A \rangle} = 0$, we obtain
\begin{align}
\begin{split}
    \|A\|^2 &= \|UU^H A + (I - UU^H) A\|^2 \\
    &= \|UU^H A\|^2 + \|(I - UU^H) A\|^2
       + \langle UU^H A,\, (I - UU^H) A \rangle + \langle (I - UU^H) A,\, UU^H A \rangle \\
    &= \|UU^H A\|^2 + \|(I - UU^H) A\|^2 .
\end{split}
\end{align}

\end{enumerate}
\end{proof}

The proofs for HOSVD, ST-HOSVD and TTN-SVD reorder products in distinct
modes, which requires $K$ to be commutative. The proof for TT-SVD does
not require commutativity, since a TT approximation multiplies its cores
in a fixed order. For TTN-SVD, the vertices are processed one at a time
in the order of $T$, instead of level by level
as in~\cite[Thm.~11.66]{hackbusch2019}. Each iteration of
Algorithm~\ref{alg:ttn-svd} applies an orthogonal projection to the
unfolding along one internal edge, so $\|\calA-\calT\|^2$ is the sum of
the errors discarded at the $\#\mathcal{V}(T)-1$ internal edges, and each
of these is at most $\|\calA-\calT^\star\|^2$. When the root has no
external edges and two children $u$ and $w$ with $r_u=r_w$, the
iteration at the second of them discards
nothing~\cite[Rem.~11.59(c)]{hackbusch2019}, which removes one term.
As an illustration, we give the proof for $\mathrm{HOSVD}_K$ in
Section~\ref{sec: proof HOSVD}.

\subsection{Proof for $\mathrm{HOSVD}_K$}\label{sec: proof HOSVD}
Write $P_k = U_k U_k^H$. Since $U_k U_k^H U_k = U_k$, we have $P_k^H = P_k$, $P_k^2 = P_k$ and
$P_k(I - P_k) = 0$. Because $K$ is commutative, the order in which the core is formed does not matter, and combining the products in each mode,
\[
  \calT = (\calA \times_1 U_1^H \times_2 \cdots \times_d U_d^H) \times_1 U_1 \times_2 \cdots \times_d U_d
        = \calA \times_1 P_1 \times_2 \cdots \times_d P_d .
\]
Adding and subtracting $\calA \times_1 P_1 \times_2 \cdots \times_k P_k$ for $k = 1, \dots, d-1$ gives
\[
  \calA - \calT = \sum_{k=1}^d \calA_k, \qquad
  \calA_k = \calA \times_1 P_1 \times_2 \cdots \times_{k-1} P_{k-1} \times_k (I - P_k).
\]
Let $i < j$. Applying the mode-$i$ product last, $\calA_i = \calB \times_i (I - P_i)$ and
$\calA_j = \calB' \times_i P_i$, where $\calB = \calA \times_1 P_1 \times_2 \cdots \times_{i-1} P_{i-1}$
and $\calB'$ is $\calA_j$ with its mode-$i$ product removed. By the definition of the mode-$k$ TTM
and Eq.~\eqref{eq: orthogonality of projection residual} in
Proposition~\ref{prop: matrix properties},
\[
  \langle \calA_i, \calA_j \rangle
  = \langle (I - P_i)\calB_{(i)},\, P_i \calB'_{(i)} \rangle = 0 ,
\]
and hence also $\langle \calA_j, \calA_i \rangle = \overline{\langle \calA_i, \calA_j \rangle} = 0$.
Therefore
\[
  \|\calA - \calT\|^2 = \sum_{k=1}^d \|\calA_k\|^2 .
\]
Applying the mode-$k$ product first,
$\calA_k = (\calA \times_k (I - P_k)) \times_1 P_1 \times_2 \cdots \times_{k-1} P_{k-1}$. By
Eq.~\eqref{eq: pythagoras for projector} in Proposition~\ref{prop: matrix properties}, an orthogonal
projection applied in one mode does not increase the norm, so
\[
  \|\calA_k\| \le \|\calA \times_k (I - P_k)\| = \|\calA_{(k)} - P_k \calA_{(k)}\| .
\]
The $\mathrm{HOSVD}_K$ algorithm takes $U_k$ to be $r_k$ leading left singular vectors of $\calA_{(k)}$, so by
Theorem~\ref{thm: optimality of rank-k SVD approximation}, $P_k \calA_{(k)}$ is a best approximation
of $\calA_{(k)}$ of $K$-rank at most $r_k$ (it is the rank-$r_k$ truncation of an ordered SVD of $\calA_{(k)}$). Writing
$\calT^\star = \calC^\star \times_1 U_1^\star \times_2 \cdots \times_d U_d^\star$ and applying the
mode-$k$ product last, $\calT^\star_{(k)} = U_k^\star M$ for some matrix $M$, so $\calT^\star_{(k)}$
has $K$-rank at most $r_k$. Hence
\[
  \|\calA_{(k)} - P_k \calA_{(k)}\| \le \|\calA_{(k)} - \calT^\star_{(k)}\| = \|\calA - \calT^\star\| .
\]
Summing over $k$,
$\|\calA - \calT\|^2 = \sum_{k=1}^d \|\calA_k\|^2 \le d\,\|\calA - \calT^\star\|^2$.  If $\mathrm{HOSVD}_K$ is run with tolerance $\varepsilon$, then $\|\calA_{(k)}-P_k\calA_{(k)}\|^2\le\varepsilon^2\|\calA\|^2/d$ by Remark~\ref{rem: LLSV tolerance}, so $\|\calA-\calT\|\le\varepsilon\|\calA\|$.

\end{document}